\documentclass{amsart}
\usepackage{amssymb,latexsym}
\usepackage{enumerate}
\usepackage{txfonts}
\usepackage{graphicx}
\usepackage{url}

\makeatletter
\@namedef{subjclassname@2010}{\textup{2010} Mathematics Subject Classification}
\makeatother

\newtheorem{thm}{Theorem}[section]
\newtheorem{lem}[thm]{Lemma}
\theoremstyle{definition}
\newtheorem{defn}[thm]{Definition}
\newtheorem{rem}[thm]{Remark}
\newtheorem{exmp}[thm]{Example}
\newtheorem{cor}[thm]{Corollary}
\newtheorem{que}[thm]{Question}

\numberwithin{equation}{section}
\newcommand{\N}{\mathbb{N}}
\newcommand{\Z}{\mathbb{Z}}
\newcommand{\Zp}{\mathbb{Z}_+}
\newcommand{\R}{\mathbb{R}}

\newcommand{\set}[1]{\{#1\}}
\newcommand{\eps}{\varepsilon}
\newcommand{\ra}{\rightarrow}

\newcommand{\dd}{\varrho}

\DeclareMathOperator{\diam}{diam}
\DeclareMathOperator{\density}{card}
\begin{document}
\title[On almost specification and average shadowing\ldots]{On almost specification and average shadowing properties}

\author{Marcin Kulczycki}\address[M. Kulczycki]{Faculty of Mathematics and Computer Science, Jagiellonian University, ul. \L o\-jasiewicza 6, 30-348 Krak\'ow, Poland}\email{Marcin.Kulczycki@im.uj.edu.pl}
\author{Dominik Kwietniak}\address[D. Kwietniak]{Faculty of Mathematics and Computer Science, Jagiellonian University, ul. \L o\-jasiewicza 6, 30-348 Krak\'ow, Poland}\email{dominik.kwietniak@uj.edu.pl}
\author{Piotr Oprocha}\address[P.~Oprocha]{AGH University of Science and Technology, Faculty of Applied Mathematics, al. A. Mickiewicza 30, 30-059 Krak\'ow, Poland\\and Institute of Mathematics, Polish Academy of Sciences, ul. \'Sniadeckich 8, 00-956 Warszawa, Poland.} \email{oprocha@agh.edu.pl}

\begin{abstract}
In this paper we study relations between almost specification property, asymptotic average shadowing property and average shadowing property
for dynamical systems on compact metric spaces.
We show implications between these properties and relate them to other important notions such as shadowing, transitivity, invariant measures, etc.
We provide examples that compactness is a necessary condition for these implications to hold. As a consequence of our methodology we
also obtain a proof that limit shadowing in chain transitive systems implies shadowing.
\end{abstract}
\maketitle

%%%%%%%%%%%%%%%%%%%%%%%%%%%%%%%%%%%%%%%%%%%%%%%%%%%%%%%%%%%%%%%%%%%%%%%%%%%%%%%%%%%%%%%%%%%%%%%%%%%%%%%%%%%%%%%%%%%%%
\section{Introduction}

Studies on shadowing and specification, originating with the works of Anosov and Bowen, have been developing parallel
to the theory of hyperbolic systems. In a crude sense, one may say that these notions are similar. The common goal
is to find a true trajectory near an approximate one. They only differ in understanding what constitutes an approximate trajectory. For shadowing one traces a pseudo-orbit, while for specification one follows arbitrarily assembled finite pieces of orbits with a true orbit. A template definition for any generalization of shadowing (or specification) might be:
every approximate orbit can be traced by a true one. Moreover, given a quantitative methods of measuring how well an
approximate orbit resembles a true trajectory, and how closely it is traced (followed, reproduced) by an orbit of some point,
we may restate our template definition: for every $\varepsilon$ there is a $\delta$ such that every $\delta$-approximate orbit
can be traced with error not greater than $\varepsilon$. This template is a base for the generalizations
of both notions we investigate here: the almost specification property, the average shadowing property, and the asymptotic average shadowing property.
Pilyugin's book \cite{Pil} is a good general reference on shadowing and its generalizations, while Palmer's \cite{Palmer} presents their applications. Properties of systems with specification are described in \cite{DGS}.

At the end of 1980s Blank introduced the notion of average pseudo-orbits and proved that certain kinds of perturbed hyperbolic systems have the average shadowing property (see \cite{Blank-etds,Blank-88}).
Average pseudo-orbits arise naturally in the realizations of independent Gaussian random perturbations
with zero mean and in the investigations of the most probable orbits of the dynamical system with general Markov perturbations, etc. (see \cite[p. 20]{Blank2}). It is proved in \cite[Theorem 4]{Blank-etds} that if $\Lambda$ is a basic set
of a diffeomorphism $f$ satisfying Axiom A, then $f|_\Lambda$ has the average shadowing property.
The notion gained considerable attention of the scientific community, see \cite{Blank-etds,Blank-88,Blank2,Blank-book,Niu,Sakai-Rocky,Sakai-hyper,Sakai-top,Zhang}.
In \cite{Sakai-Rocky} Sakai analyzed the dynamics of diffeomorphisms satisfying the average shadowing property on a two-dimensional closed
manifold. Later, in \cite{Sakai-hyper,Sakai-top}, he compared various shadowing properties of positively expansive maps.
The results of \cite{Sakai-top} were generalized and completed in \cite{KwOp}.
In \cite{Niu} Niu proved that if $f$ has the average shadowing property and the minimal points of $f$ are dense in $X$, then $f$ is weakly mixing
and totally strongly ergodic.

The next property we consider is the \emph{asymptotic average shadowing} introduced by Gu in \cite{Gu}. Gu followed the same scheme as
Blank, but with limit shadowing instead of shadowing as the starting point for generalization.
The asymptotic average shadowing property was examined, inter alia,  in \cite{KuOp1,KuOp2}. It was proved
that there is a large class of systems with the asymptotic average shadowing property, including all mixing
maps of the unit interval and their Denjoy extensions.

More recently, Climenhaga and Thompson (\cite{CT2,T}), inspired by the work of Pfister and Sullivan \cite{PS},
examined some properties of systems with the almost specification property, which in turn
generalizes the notion of specification. As all beta shifts have the almost specification property their results apply to those
important symbolic systems.

We believe that techniques and notions described above deserve deeper study and the results scattered through the literature should be put into a unified framework.
Therefore our main goal is to explore the general properties of systems possessing generalized shadowing and/or specification in the abstract setting.

It follows from our work that these generalizations are related much closer than the original notions. Specification is well known to imply
almost specification, and there are examples of systems with almost specification, but without specification. We show that the almost specification property implies asymptotic average shadowing property (Theorem \ref{thm:almost-spec-implies-aasp}), which in turn implies average shadowing property (Theorem \ref{thm:aasp-implies-asp}). We do not know whether the converse implications are true when the space under consideration is compact, which is the usual setting (Question \ref{q:converse_implications}). Nevertheless, we are able to give examples of maps on noncompact metric spaces that either have the average shadowing property, but not the asymptotic average shadowing property (Example \ref{example1}), or have the asymptotic average shadowing property, but not the average shadowing property (Example \ref{example2}).

We also explore recurrence properties of dynamical systems with the average shadowing property. We prove that every dynamical system with the almost specification property and full invariant measure (a measure which is positive on each nonempty open set) has a dense set of minimal points (Theorem \ref{thm:dense_minimal}) and that every dynamical system with the average shadowing property and full invariant measure is topologically weakly mixing (Theorem \ref{thm:measure+asp-imply-wm}). Moreover, we show that if the supports of invariant measures are not dense, then there is no recurrence property that is implied by the average shadowing property (see Section~\ref{sec:examples} for multiple examples of this kind). In Section~\ref{sec:sec_measure_center} we prove that $f$ has the average shadowing property (almost specification property, asymptotic average shadowing property) provided $f$ restricted to its measure center (the closure of the union of all supports of $f$-invariant measures) has the average shadowing property (almost specification property, asymptotic average shadowing property).

As a byproduct of our methodology we obtain a little bit surprising result (Theorem~\ref{limsh_chaintrans}) that limit shadowing in chain transitive systems implies shadowing (and so transitivity).
It was proved recently by Lee and Sakai \cite{LiSakai} that expansive systems with shadowing have limit shadowing (in fact they have the so-called strong limit shadowing property, which is
stronger than shadowing and limit shadowing together as shown in \cite[Example 3.5]{BGO}). There are also known examples of systems with limit shadowing but without
shadowing (see \cite[Example 1.21]{Pil}, which may be generalized to a large class of homeomorphisms on $[0,1]$).
By the above evidence it is natural to expect that limit shadowing is a weaker property than shadowing property. In fact Theorem~\ref{limsh_chaintrans} and Theorem~\ref{thm:aasp-implies-asp} are completely opposite to the results that the authors were aiming to prove when they started working on the topics included in this paper.

%%%%%%%%%%%%%%%%%%%%%%%%%%%%%%%%%%%%%%%%%%%%%%%%%%%%%%%%%%%%%%%%%%%%%%%%%%%%%%%%%%%%%%%%%%%%%%%%%%%%%%%%%%%%%%%%%%%%%
%%%%%%%%%%%%%%%%%%%%%%%%%%%%%%%%%%%%%%%%%%%%%%%%%%%%%%%%%%%%%%%%%%%%%%%%%%%%%%%%%%%%%%%%%%%%%%%%%%%%%%%%%%%%%%%%%%%%%
\section{Preliminaries}\label{prelims}
Let $\N=\{0,1,\ldots\}$ denote the set of natural numbers and let $\Zp, \Z_-$ denote the set of positive and negative  integers, respectively.
For any $A\subset\mathbb{N}$ and a set $J\subset\N$ we let $\density(A\mid J)$ denote the cardinality of $A\cap J$. Given $n>0$ we write $\density(A\mid n)$ for $\density(A\mid J)$ with $J=\{0,1,\ldots,n-1\}$.
We define the \emph{lower} and \emph{upper asymptotic density} of a set $J\subset \N$ as
\[
d_*(J)=\liminf_{n\ra\infty}\frac{1}{n}\density(J\mid n)\quad\text{and}\quad d^*(J)=\limsup_{n\ra\infty}\frac{1}{n}\density(J\mid n),
\]
respectively. The set $J$ is said to have the \emph{asymptotic density} $d(J)$ provided that $d_*(J)=d^*(J)=d(J)$.
A set $J\subset\N$ has \emph{positive upper Banach density} if and only if for some sequence of integers $\{k_n\}_{n=1}^\infty$ we have
\[
\limsup_{n\ra\infty}\frac{1}{n}\density(J\mid \{k_n,k_n+1,\ldots,k_n+n-1\})>0.
\]
A set $A\subset\N$ is \emph{thick} if it contains arbitrarily long  sequences of consecutive natural numbers.

Let $J\subset \N$ be a set such that $\N\setminus J$ is unbounded. Let $\set{a_i}_{i=0}^\infty$ be a sequence of real numbers. If there is $b\in \R$ such that the sequence obtained from $\set{a_i}_{i=0}^\infty$ by deleting the terms with indices from $J$ has limit $b$, then we write
\[\lim_{i\not\in J} a_i=b.\]

We record the following Lemma for further reference.
\begin{lem}[{\cite{Walters}{, Thm 1.20}}]\label{Walters}
Let $\{a_i\}_{i=0}^\infty$ be a bounded sequence of non-negative real numbers. The following conditions are equivalent:
\begin{enumerate}
\item $\lim_{n\rightarrow\infty}\frac{1}{n}\Sigma_{i=0}^{n-1}a_i=0$, \label{lemPW:1}
\item There exists a set $J\subset \N$ such that $d(J)=0$ and $\lim_{n\not\in J}a_n=0$.\label{lemPW:2}
\end{enumerate}
\end{lem}

By ``a map'' we always understand a continuous map. Given a metric space $(X,\dd)$ and a map $f\colon X\ra X$ we call the pair $(X,f)$ a \emph{dynamical system}. Throughout this paper, unless stated otherwise, we assume that a dynamical system $(X,f)$ on a compact space $X$ is given.

A dynamical system $(Y,g)$ is a \emph{factor} of a dynamical system $(X,f)$ if there is a continuous surjection $\pi\colon X\ra Y$ such that $\pi \circ f=g\circ \pi$.

The \emph{orbit of $x\in X$} is the set $\set{f^n(x)\: :\: n\in \N}$. We say that $f$ is \emph{minimal} if the orbit of every $x\in X$ is a dense subset of $X$. The map $f$ is \emph{transitive} (respectively \emph{mixing}) if for any pair of nonempty open sets $U, V \subset X$ there exists $n > 0$ ($N>0$) such that $f^n(U) \cap V\neq \emptyset$ (for all $n\ge N$, respectively). We say that $f$ is \emph{totally transitive} if $f^n$ is transitive for all $n\geq 1$. The map $f$ is \emph{weakly mixing} if $f\times f$ is transitive on $X\times X$.

A pair $(x,y)\in X\times X$ is \emph{proximal} if
\[
\liminf_{n\ra\infty}\dd(f^n(x),f^n(y))=0.
\]
A pair $(x,y)\in X\times X$ is \emph{distal} if it is not proximal. We say that a pair $(x,y)\in X\times X$ is \emph{diagonal} if $x=y$. A dynamical system $(X,f)$ is
\begin{itemize}
\item \emph{proximal} if all pairs in $X\times X$ are \emph{proximal},
\item \emph{distal} if all non-diagonal pairs in $X\times X$ are \emph{distal}.
\end{itemize}
Compactness of $X$ implies that these properties are independent of the choice of equivalent metric for $X$.

The map $f$ is \emph{equicontinuous} if for every $\eps > 0$, there exists a $\delta > 0$ such that $\dd(f^n(x), f^n(y)) < \eps$ for all $n>0$ and all $x,y\in X$ such that $\dd(x,y) < \delta$. Every dynamical system $(X,f)$ has a \emph{maximal equicontinuous factor} $(X_{\text{eq}},f_{\text{eq}})$. That is, $(X_{\text{eq}},f_{\text{eq}})$ is equicontinuous, and every equicontinuous factor of $(X,f)$ is a factor of $(X_{\text{eq}},f_{\text{eq}})$.

Let $I=\{0,\ldots,n\}$ or $I=\N$. A sequence $\set{x_i}_{i\in I}$ is called a \emph{$\delta$-pseudo-orbit} of $f$ (from $x_0$ to $x_n$ and of \emph{length} $n$ if $I=\{0,\ldots,n\}$) if $\dd(f(x_{i-1}),x_{i})<\delta$ for every positive $i\in I$.
A sequence $\{x_i\}_{i=0}^\infty\subset X$ is an \emph{asymptotic pseudo-orbit} of $f$ provided that
$$
\lim_{n\rightarrow\infty}\dd(f(x_i),x_{i+1})=0.
$$
We say that $f$ is \emph{chain transitive} if for any $\delta>0$ and for any points $x,y\in X$ there is a  $\delta$-pseudo-orbit from $x$ to $y$. A map $f$ is \emph{chain mixing} if for any $\delta > 0$ and any $x, y \in X$ there is an integer $N>0$ such that for every $n\geq N$ there is a finite $\delta$-pseudo-orbit from $x$ to $y$ of length $n$. This is equivalent to $f\times f$ being chain transitive (see \cite{RichWise}). Given sets $A,B\subset X$ we define
the \emph{set of transition times from  $A$ to $B$} by
\[
N(A,B)=\{n>0: f^n(A)\cap B \neq\emptyset\}.
\]
If $x\in X$, then $N(x,B)=N(\{x\},B)=\{n>0: f^n(x)\in B\}$ denotes the
\emph{set of visiting times}. Note that there are no universally accepted names
for the sets $N(A,B)$ and $N(x,B)$. Some authors (see, e.g., \cite{LiJian}) prefer to call $N(A,B)$
the set of \emph{hitting times of $A$ and $B$}, and $N(x,B)$ the set of times
\emph{$x$ enters into $B$}, respectively.

We let $M(X)$ denote the space of all Borel probability measures on $X$. We say that a measure $\mu\in M(X)$ is \emph{invariant for} $f\colon X\ra X$ if $\mu(A)=\mu(f^{-1}(A))$ for any Borel set $A\subset X$.

The classical Krylov-Bogolyubov theorem guarantees that every compact dynamical system $(X,f)$ has at least one invariant measure. If the system has exactly one invariant measure then we say that it is \emph{uniquely ergodic}.

We say that a set $Y\subset X$ is \emph{measure saturated} if it is contained in the closure of the union of the topological supports of invariant measures, or equivalently, if for every open set $U$ such that $U\cap Y\neq\emptyset$ there exists an invariant measure $\mu$ such that $\mu(U)>0$. The \emph{measure center of $f$} is the largest measure saturated set.

We say that an open set $U\subset X$ is \emph{universally null} if $\mu(U)=0$ for any invariant measure $\mu$.
The \emph{measure center} of $f$ is the complement of the union of all universally null sets.

Given a nonempty Borel set $A$ and $n>0$ we define
\[
\eta(n,A)=\max_{x\in X}\density(N(x,A)\mid n).
\]
We note that $\eta(n+m,A)\leq \eta(n,A)+\eta(m,A)$ for all $n,m>0$, hence $\{\eta(n,A)\}_{n=1}^\infty$ is a subadditive sequence, and we may define
the \emph{visit frequency in $U$} as
\[
\xi(A)=\lim_{n\to\infty}\frac{\eta(n,A)}{n}=\inf_{n>0}\frac{\eta(n,A)}{n}.
\]

The following lemma follows from the ergodic theorem, but here we present a direct topological proof inspired by \cite{HLY}.

\begin{lem}\label{lem:density}
If $(X,f)$ is a compact dynamical system, then for every nonempty open set $U\subset X$ there exists a point $x\in X$
such that
\[
d(N(x,U))=\xi(U).
\]
\end{lem}
\begin{proof}
First observe that if there is a point $x\in X$ such that $N(x,U)$ has positive upper Banach density then $\xi(U)>0$.
Therefore $\xi(U)=0$ implies that $N(x,U)$ has upper Banach density zero for every $x\in X$, hence
$d(N(x,U))=0$ for every $x\in X$.

Next, let us assume that $\xi(U)>0$.
For every $n>0$ let $x_n$ be a point in $X$ such that
\[
\density(N(x_n,U)\mid n)=\max_{x\in X}\density(N(x,U)\mid n)=\eta(n,U).
\]
We claim that for each integer $n>0$ there exists a point $y_n\in X$ such that
\begin{equation}\label{eq:star}
\xi(U)-1/n%\frac{1}{n}
\leq (1/j) %\frac{1}{j}
\cdot\density(N(y_n,U)\mid j) \quad\text{for}\,1\leq j \leq n.
\end{equation}
For the proof of the claim, assume on the contrary that \eqref{eq:star} does not hold for some $k>0$.
Then $\xi(U)-1/k>0$. Put $m=k^2+1$ and let
$z=x_m$ be a point defined above so that $\density\big(N(z,U)\mid m\big)=\eta(m,U)$.
As we assumed that our claim fails we can find a strictly increasing sequence
of integers $\{l(s)\}_{s=0}^\infty$ such that $l(0)=0$, $0<\lambda_j=l({j})-l(j-1)\leq k$, and
\begin{equation}\label{eq:i}
\frac{1}{\lambda_j}
\density\big(N(f^{l(j-1)}(z),U\big)\mid \lambda_j))
< \xi(U)-\frac{1}{k},
\end{equation}
for every $j=1,2,\ldots$. Let $t>0$ be such that $ l(t) \leq m < l(t+1)$. Then
\begin{eqnarray}\label{eq:ii}
&&\density\big(N(z,U)\mid m\big)=\\
&&\quad\quad\quad=\bigg(\sum_{j=0}^{t}\density\big(N(f^{l(j-1)}(z),U)\mid \lambda_j\big)\bigg) + \density\big(N(f^{l(t)}(z)\mid m-l(t)\big)\nonumber.
\end{eqnarray}
Then by \eqref{eq:i} and \eqref{eq:ii} we obtain that
\begin{align*}
m\xi(U)  &\leq \eta(m,U)= \density\big(N(z,U)\mid m\big) = \\
         &= \bigg(\sum_{j=0}^{t}\density\big(N(f^{l(j-1)}(z),U\mid \lambda_j\big)\bigg) + \density\big(N(f^{l(t)}(z)\mid m-l(t)\}\big)\\
         &\leq \bigg[\sum_{j=0}^{t}\lambda_j \bigg(\xi(U)-\frac{1}{k}\bigg)\bigg] + k < m\bigg(\xi(U)-\frac{1}{k}\bigg)+k = m\xi(U) - \frac{1}{k}  <m\xi(U),
\end{align*}
which is a contradiction, so our claim holds.

By the claim, for each integer $n>0$ we can find a point $y_n$ such that \eqref{eq:star} holds. Since $X$ is compact,
without loss of generality we may assume that
$\{y_n\}_{n=1}^\infty$ converges to some point $x\in X$. Observe that for every $k>0$ and every $n\geq k$ we have
that
$\xi(U)-1/k%\frac{1}{k}
\leq (2/k)%\frac{1}{k}
\density(N(y_n,U)\mid k)$, which implies that
\[
\xi(U)-1/k%\frac{1}{k}
\leq (1/k)\cdot%\frac{1}{k}
\density(N(x,U)\mid k)\leq
 (1/k)\cdot\eta(k,U).%\frac{\eta(k,U)}{k}.
\]
Passing to the limit with $k\to\infty$ we obtain $\xi(U)\leq d_*(N(x,U))\leq d^*(N(x,U))\leq \xi(U)$, which finishes the proof.
\end{proof}

We can now characterize universally null open sets as the sets with visit frequency $\xi$ equal to zero.

\begin{thm}\label{thm:xi_univnull}
Let $(X,f)$ be a compact dynamical system.
An open set $U\subset X$ is universally null if and only if $\xi(U)=0$, or equivalently $d(N(x,U))=0$ for every $x\in X$.
\end{thm}
\begin{proof}
If $U$ is not universally null, then $\mu(U)>0$ for some $f$-invariant measure $\mu\in M(X)$.
Using ergodic decomposition \cite[p. 153]{Walters} we get
an ergodic measure $\mu_e\in M(X)$ such that $\mu_e(U)>0$. By Pointwise Ergodic Theorem
\cite[Thm. 1.14]{Walters}
there is a point $x\in X$ such that
$d (N(x,U))=\mu_e(U)>0$. Therefore $\xi(U)>0$.

The proof of the other implication follows the same lines as the second part of \cite[Lemma 3.17]{Fur}. We present it for completeness. Assume that $\xi(U)>0$. By Lemma \ref{lem:density} there is a point $x\in X$ and $0<\eps <\xi(U)$ such that
\[
\density\big( N(x,U) \mid n \big) \geq n\eps
\]
for all sufficiently big $n$.
Observe that if we choose any continuous function $F\colon X \to \R$ then there is an increasing function $\sigma\colon \N \to \N$ such that the limit
\[
L(F)=\lim_{n\to \infty}
\frac{1}{\sigma(n)}\sum_{i=0}^{\sigma(n)-1}
F(f^i(x))
\]
exists. This implies that for any sequence of continuous functions $F_k\colon X\ra \R$ we can use a diagonal procedure to find an increasing function $\sigma\colon \N \to \N$ such that $L(F_k)$ exists for every $k$.
As the space $\mathcal{C}(X)$ of all continuous functions from $X$ to $\R$ with the supremum metric is separable
we fix a sequence $\set{F_k}_{k=0}^\infty$ dense in $C(X)$ and choose a sequence $\sigma$ as above.
Then it is elementary to check that with this particular $\sigma$ the number $L(F)$
is well defined for every continuous function $F\colon X \to \R$.
Hence we defined a linear functional $L$ on $\mathcal{C}(X)$.
By the Riesz Representation Theorem there is a measure $\mu\in M(X)$ such that $L(F)=\int F d\mu$, and since $L(F)=L(F\circ f)$, the measure  $\mu$ must be an invariant measure for $f$. But
\[
\mu(U)=\int \chi_U d\mu = L(\chi_U) = \lim_{n\to \infty}
\frac{1}{\sigma(n)}\sum_{i=0}^{\sigma(n)-1}
\density\big(N(x,U)\mid \sigma(n)\big) \geq \eps > 0.
\]
This concludes the proof.
\end{proof}

\begin{cor}\label{cor:measure_cent}
If a set $A$ contains the measure center
of a compact dynamical system $(X,f)$,
then for every $\varepsilon >0$ there exists $N\in\N$ such that for every $x\in X$
and $n\ge N$
we have
$$
\frac{1}{n}\#\set{0\leq i<n \;:\; \dd(f^i(x),A) <\varepsilon} >1-\varepsilon.
$$
\end{cor}
\begin{proof}
Fix any $\eps>0$ and let $U=\set{x : \dd(x,A)>\eps/2}$. Since $A$ contains the measure center, $U$ is a universally null open set. By Theorem~\ref{thm:xi_univnull}
we obtain that $\xi(U)=0$ and so for all sufficiently large $n$ we have
\begin{align*}
\max_{x\in X}\frac{1}{n}\#\set{0\leq i<n \;:\; \dd(f^i(x),A) \geq \eps} &\leq \max_{x\in X}\frac{1}{n}\density(N(x,U)\mid n)\\
 &=\frac{1}{n}\eta(n,U)\leq \xi(U)+\frac{\eps}{2}<\eps.
\end{align*}
\end{proof}

The following definitions were introduced by Blank in \cite{Blank-etds,Blank-88}:

\begin{defn}\label{def:aver_po}
Given $\delta>0$ we say that a sequence $\set{x_n}_{n=0}^\infty$ is a \emph{$\delta$-average-pseudo-orbit of $f$} if there is an integer $N>0$ such that for all $n\geq N$ and $k\geq 0$ the following condition is satisfied:
$$
\frac{1}{n}\sum_{i=0}^{n-1} \dd(f(x_{i+k}),x_{i+k+1})<\delta.
$$
\end{defn}

\begin{defn}
Given $\eps>0$ and $y\in X$ we say that a sequence $\set{x_n}_{n=0}^\infty$ is \emph{$\eps$-shadowed in average by $y$} if
$$
\limsup_{n\ra\infty} \frac{1}{n}\sum_{i=0}^{n-1} \dd(f^i(y),x_i)<\eps.
$$
\end{defn}

\begin{defn}
We say that $f$ has the \emph{average shadowing property} if for every $\eps>0$ there is $\delta>0$ such that every $\delta$-average-pseudo-orbit of $f$ is $\eps$-shadowed in average by some point in $X$.
\end{defn}

The next three definitions were introduced for the first time by Gu in \cite{Gu}:

\begin{defn}
The sequence $\{x_i\}_{i=0}^\infty\subset X$ is an \emph{asymptotic average pseudo-orbit of $f$} provided that $$ \lim_{n\rightarrow\infty}\frac{1}{n}\sum_{i=0}^{n-1}\dd(f(x_i),x_{i+1})=0.$$
\end{defn}

\begin{defn}
The sequence $\{x_i\}_{i=0}^\infty\subset X$ is \emph{asymptotically shadowed in average} by the point $y\in X$ provided that $$\lim_{n\rightarrow\infty}\frac{1}{n}\sum_{i=0}^{n-1}\dd(f^i(y),x_i)=0.$$
\end{defn}

\begin{defn}
The map $f$ has the {\it asymptotic average shadowing property} provided that every asymptotic average pseudo-orbit of $f$ is asymptotically shadowed in average by some point in $X$.
\end{defn}

Pfister and Sullivan \cite{PS} have introduced a property called the \emph{$g$-almost product property}.
Subsequently Thompson proposed renaming this property as the almost specification property.
The only difference between the approach in \cite{PS} and the one presented in \cite{T}
is that the mistake function $g$ can depend also on $\eps$ (in \cite{PS} function $g$ depends on $n$ alone).
The almost specification property can be verified for every $\beta$-shift (see \cite{PS2}).
We follow Thompson, so the version we use here is a priori
weaker than that of Pfister and Sullivan. First we need a few auxiliary definitions.

\begin{defn}
Let $\eps_0 > 0$. A function $g\colon \Zp \times (0, \eps_0] \ra \mathbb{N}$ is
called a \emph{mistake function} if for all $\eps\in(0, \eps_0]$ and all $n \in \Zp$, we have $g(n,\eps) \leq  g(n+1,\eps)$ and
\[
\lim_{n\to \infty}
\frac{g(n, \eps)}{n}= 0.
\]
Given a mistake function $g$, if $ \eps > \eps_0$, then we define $g(n, \eps) = g(n, \eps_0)$ .
\end{defn}

\begin{defn}
Let $g$ be a mistake function and let $\eps>0$. For $n$ large enough for the inequality $g(n, \eps) < n$ to hold we
define the \emph{set of $(g; n, \eps)$ almost full subsets of $\set{0,\ldots,n-1}$} as the family
$I(g; n, \eps)$ consisting of subsets of $\set{0,1,\ldots,n-1}$ with at least $n - g(n,\eps)$ elements, that is,
\[
I(g; n, \eps) := \{\Lambda\subset \set{0,1,\ldots, n - 1} : \#\Lambda \geq n - g(n,\eps)\}.
\]
\end{defn}

\begin{defn}\label{def:lambda:asp}
For a finite set of indices $\Lambda\subset \set{0,1,\ldots, n - 1}$, we define the \emph{Bowen distance between $x,y\in X$ along $\Lambda$} by
$\dd_\Lambda(x, y) = \max\set{\dd(f^j(x), f^j(y)) : j \in\Lambda}$ and the \emph{Bowen ball (of radius $\eps$ centered at $x\in X$) along $\Lambda$}
by $B_\Lambda(x, \eps) = \set{y \in X : \dd_\Lambda(x, y) < \eps}$.
When $g$ is a mistake function and $(n,\eps)$ is such that $g(n,\eps) < n$, we define for $x\in X$ a \emph{$(g;n,\eps)$-Bowen ball of radius $\eps$, center $x$, and length $n$} %with $g(n, \eps)$ mistakes}
by
\[
B_n(g; x, \eps) := \bigg\{y \in X : y \in B_\Lambda(x, \eps) \text{ for some }\Lambda\in I(g;n,\eps)\bigg\}
=
\bigcup_{\Lambda\in I(g;n,\eps)}
B_\Lambda(x, \eps).
\]
\end{defn}

With the above notation in mind we are in position to state the definition of the almost specification property.

\begin{defn}
A continuous map $f\colon X \ra X$ has the \emph{almost specification
property} if there exists a mistake function $g$ and a function $k_g\colon (0,\infty) \to \N$ such that
for any $m\geq 1$, any $\eps_1,\ldots,\eps_m > 0$, any points $x_1, \ldots, x_m \in X$, and any integers
$n_1 \geq k_{g}(\eps_1),\ldots,n_m \geq k_{g}(\eps_m)$ setting $n_0=0$ and
$$
l_j=\sum_{s=0}^{j-1}n_s,\,\text{for }j=1,\ldots,m
$$
we can find a point $z\in X$ such that for every $j=1,\ldots,m$ we have
$$
f^{l_j}(z)\in B_{n_j}(g;x_j,\eps_j).
$$
In other words,
the appropriate part of the orbit of $z$ $\eps_j$-traces with at most $g(\eps_j,n_j)$ mistakes the orbit of $x_j$.
\end{defn}

\begin{rem}
If the system $(X,f)$ has almost specification with a  mistake function $g$ we may (and usually will) assume that
$k_g(\eps)=n$ implies $g(m,\eps)<m$ for $m\ge n$.
\end{rem}

The function $g$ as above can be interpreted as follows. The integer $g(n, \eps)$ tells
us how many mistakes may occur when we use the almost specification
property to $\eps$-shadow an orbit of length $n$.

Note that Pfister and Sullivan use a slightly different notion of a blowup function instead of the mistake function defined above. The blowup function is not allowed to depend on $\eps$. An example of a function which is a mistake function under definition proposed by Thompson but is not
considered by Pfister and Sullivan is $g(n, \eps) = \lfloor\eps^{-1}\log n\rfloor$. Therefore the almost specification property of Thompson is slightly more general than the $g$-almost product property of Pfister and Sullivan.

Note that, while in our setting the space $X$ is compact, the definitions of the (almost) specification property and the two generalized shadowing properties remain meaningful without this assumption. But in a noncompact setting none of these properties (specification, almost specification, (asymptotic) average shadowing) is an invariant for topological conjugacy, as can be seen from
the example in \cite[Section 7]{KU} and Theorem \ref{thm:summary} below. The reader can find several comments on the specification property and its relationship to the (asymptotic) average shadowing property in \cite{KuOp1,KuOp2, KwOp}.
Moreover, note the following.

\begin{exmp} Let $X = \{a, b\}$ be any two points set with the discrete metric $\rho$ and let $f$ be such that
$f (a) =b$, $f (b) =b$. Then $f$ has the (almost) specification property and (asymptotic) average shadowing property.
\end{exmp}

%%%%%%%%%%%%%%%%%%%%%%%%%%%%%%%%%%%%%%%%%%%%%%%%%%%%%%%%%%%%%%%%%%%%%%%%%%%%%%%%%%%%%%%%%%%%%%%%%%%%%%%%%%%%%%%%%%%%%
%                 section CONNECTIONS BETWEEN ALMOST SPECIFICATION AND AVERAGE SHADOWING
%%%%%%%%%%%%%%%%%%%%%%%%%%%%%%%%%%%%%%%%%%%%%%%%%%%%%%%%%%%%%%%%%%%%%%%%%%%%%%%%%%%%%%%%%%%%%%%%%%%%%%%%%%%%%%%%%%%%%

\section{Connections between almost specification and average shadowing}

\subsection{Chain mixing}

It turns out that if $f$ is surjective, then chain mixing accompanies both the average shadowing and the almost specification properties, as evidenced by the next two lemmas. Note that surjectivity of $f$ is necessary by

\begin{lem}\label{lem:surj-asp-implies-chain-mix}
Let $(X,f)$ be a compact dynamical system.
If $f$ is surjective and has the average shadowing property, then $f$ is chain mixing.
\end{lem}
\begin{proof}
It was proved in \cite{Niu} that if $f$ has the average shadowing property, then so does $f^n$ for every $n>0$.
By \cite{RichWise}, if $f^n$ is chain transitive for every $n>0$, then it is chain mixing.
Therefore it is enough to prove that $f$ is chain transitive.

For simplicity of calculations let us assume that $\diam(X)\leq 1$.
Fix any $\eps>0$ and any points $x,y\in X$. Let $\delta$ be provided by the average shadowing property for $\eps/2$.
Let $n_0\ge 2$ be such that $2/n_0<\delta$ and let $n_{i}=2^{i} n_0$ for $i\ge 1$.
For $j=0,1,2,\ldots$ we use surjectivity to fix a point $y_j\in f^{-n_{2j+1}+1}(y)$ and define a sequence
\begin{multline*}
\xi = (x,f(x),\ldots, f^{n_0-1}(x), y_0, f(y_0),\ldots f^{n_1-1}(y_0),x,f(x),\ldots, f^{n_2-1}(x),\\
y_1, f(y_1),\ldots f^{n_3-1}(y_1),
\ldots,x, f(x),\ldots,f^{n_{2j}-1}(x),y_{j},f(y_{j}),\ldots, f^{n_{2j+1}-1}(y_{j}),\ldots).
\end{multline*}
Let $l(j)=(2^j-1)n_0$. By the definition of $\xi$ we have
$\xi_{l(2i)}=x$, $\xi_{l(2i+2)-1}=y$ and $\xi_{l(2i+1)}=y_i$ for all $i=0,1,2,\ldots$
Therefore $\xi_{l(j)},\ldots,\xi_{l(j+1)-1}$ is the initial segment of length $n_j$ of orbit of $x$ if $j$ is even, and of $y_k$, where $k=(j-1)/2$, if $j$ is odd.

Notice that if we fix any $i\geq 0$ then in the sequence $\xi_i,\ldots, \xi_{i+n_0}$ there is at most one position $j$ with $\dd(f(\xi_j),\xi_{j+1})>0$. Therefore for every $k>n_0$ we have
$$
\frac{1}{k}\sum_{j=i}^{i+k-1} \dd(f(\xi_j),\xi_{j+1})\leq \frac{1}{k}\left(\frac{k}{n_0}+1\right)\diam(X) \leq \frac{2}{n_0}<\delta,
$$
which shows that $\xi$ is an $\delta$-average-pseudo-orbit. Let $z\in X$ be a point which $\eps$-shadows $\xi$ in average. This implies that there are
$p,q\in\N$ such that $\dd(f^p(x),f^q(z))<\eps$ and $r,s,t\in\N$ such that $q<  l(2t)\le r$, $s<n_{2t+1}$, and $\dd(f^r(z),f^s(y_t))<\eps$, for otherwise we would have
$$
\frac{1}{l(j)}\sum_{i=0}^{l(j)-1}\dd(f^i(z),\xi_{i})\geq\frac{2^{j-1}n_0}{(2^{j}-1)n_0}\eps
\geq\frac{\eps}{2}
$$
for every sufficiently large $j\in\N$ of some fixed parity (odd or even). We conclude that for some $p,q,r,s,t$ with $q<r-1$ and $s<n_{2t+1}$ the sequence
$$
x,f(x),\ldots, f^{p-1}(x),f^{q}(z),\ldots,f^{r-1}(z),f^{s}(y_t),\ldots,y
$$
is a finite $\eps$-pseudo-orbit from $x$ to $y$, which completes the proof.
\end{proof}

\begin{lem}\label{lem:surj-almost-spec-implies-chain-mix}
Let $(X,f)$ be a compact dynamical system.
If $f$ is surjective and has the almost specification property, then $f$ is chain mixing.
\end{lem}
\begin{proof}
Let $x,y\in X$ and fix $\eps>0$.  Use the almost specification property to obtain the constant $N=k_{g}(\eps)$. %Choose $N>0$ large enough for the inequality $g(n,\eps)\leq n-1$ to hold for all $n\geq N$.

We will show that for each $n\geq 2N+2$ there is an $\eps$-chain of length $n$ from $x$ to $y$. To this end we fix $n\geq 2N+2$. Using surjectivity of $f$ we find $y_0\in X$ such that $f^{n-N}(y_0)=y$. By the almost specification property there is a point $z\in X$ such that $z\in B_N(f(x);g,\eps)$ and $f^N(z)\in B_{N}(y_0;g,\eps)$.
Equivalently, there is $0\leq s \leq N-1$ such that
$\dd(f^{s+1}(x),f^s(z))<\eps$ and there is $N\leq t \leq 2N$ such that
$\dd(f^t(z),f^{t-N}(y_0))<\eps$. Therefore the sequence
$$
x,f(x),\ldots, f^{s}(x),f^s(z),f^{s+1}(z),\ldots,f^{t-1}(z),f^{t-N}(y_0),f^{t-N+1}(y_0),\ldots,y
$$
is an $\eps$-pseudo-orbit of length $n$ from $x$ to $y$, which completes the proof.
\end{proof}

\begin{lem}\label{lem:apo-from-aapo}
Let $(X,f)$ be a compact dynamical system.
If $f$ is chain mixing and $\set{x_j}_{j=0}^\infty$ is an asymptotic average pseudo-orbit of $f$,
then there exists an asymptotic pseudo-orbit $\set{y_j}_{j=0}^\infty$ of $f$ such that the set $\set{j : x_j\neq y_j}$ has asymptotic density zero.
\end{lem}

\begin{proof}
By chain mixing of $(X,f)$, for every $k$ there exists an integer $M_k$ such that for any points $x,y\in X$ there is a $1/2^k$-chain of length $M_k+1$ from $x$ to $y$. Without loss of generality we may assume that $M_{k+1}$ is a multiple of $M_k$ for every $k$.

Let $\set{x_j}_{j=0}^\infty$ be an asymptotic average pseudo-orbit. By Lemma~\ref{Walters} there exists a set $J$ such that
$d(J)=0$ and
\begin{equation}\label{notin}
\lim_{j\not\in J} \dd(f(x_j),x_{j+1})=0.
\end{equation}
Since $d(J)=0$, there exists a strictly increasing sequence $\set{n_k}_{k=1}^\infty$ such that for every $k$ we have $M_{k+1}$ divides $n_k$
and
$$
M_k \cdot \frac{\density(J\mid n)}{n}< \frac{1}{2^k}
$$
for every $n>n_k$. By \eqref{notin} we may assume (increasing $n_k$ if necessary) that if $j\not\in J$ and $j\geq n_k$ then $\dd(f(x_j),x_{j+1})<1/2^k$.
Now we define a set $J'$ and a sequence $\set{y_j}_{j\in J'}$ in the following way: for every $k$ and for every $s$ such that $[sM_k,(s+1){M_{k}})\subset [n_k,n_{k+1})$
if
$J\cap [sM_k,(s+1){M_{k}})\neq\emptyset$
then we include the set $[sM_k,(s+1)M_k]\cap\N$ in $J'$ and we define $\set{y_j}_{j=sM_k}^{(s+1)M_k}$ as a $1/2^k$-chain
from $x_{sM_k}$ to $x_{(s+1)M_k}$ of length $M_{k}+1$. Note that $y_{sM_k}=x_{sM_k}$ and $y_{(s+1)M_k}=x_{(s+1)M_k}$.

Let $J'$ and $\set{y_j}_{j\in J'}$ be obtained by the above procedure. For $j\not\in J'$ we put $y_j=x_j$.
First note that for $j\geq n_k$ we have $\dd(f(y_j),y_{j+1})<1/2^k$ and hence $\set{y_j}_{j=0}^\infty$ is an asymptotic pseudo-orbit. Furthermore, if we fix any $n>n_1$
then there is $k>0$ such that $n_k< n \leq n_{k+1}$ and then
\begin{eqnarray*}
\frac{1}{n}\density(J'\mid n)&=&\frac{1}{n}\density(J'\mid n_1)+\frac{1}{n}\sum_{s=1}^k\density(J'\cap [n_s,n_{s+1})\mid n)\\
&\leq&\frac{1}{n}\density(J'\mid n_1)+\frac{1}{n}\sum_{s=1}^k M_s\density(J \cap [n_s,n_{s+1})\mid n)\\
&\leq&\frac{1}{n}\density(J'\mid n_1)+M_k\cdot\frac{\density(J\mid n)}{n}\\
&<&\frac{1}{n}\density(J'\mid n_1)+\frac{1}{2^k}.
\end{eqnarray*}
This shows that $d(J')=0$. The proof is completed by noting that
$\set{j : x_j\neq y_j}\subset J'$.
\end{proof}

%%%%%%%%%%%%%%% subsection ALMOST SPECIFICATION IMPLIES ASYMPTOTIC AVERAGE SHADOWING %%%%%%%%%%%%%%%%%%%%%%%%%%%%%%%%%%
\subsection{Almost specification implies asymptotic average shadowing}

The similarities between almost specification and asymptotic average shadowing are rather vague and it is not obvious whether either one implies the other. In this section we show that almost specification implies asymptotic average shadowing for surjective maps.

\begin{lem}\label{lem:inf_almostspec}
Let $X$ be compact, let $f\colon X \ra X$ be a map with the almost specification property, and let $g$ be its mistake function. Assume we are given
\begin{enumerate}
\item a sequence of positive real numbers $\set{\eps_i}_{i=1}^\infty$,
\item a sequence of points $\set{x_i}_{i=1}^\infty\subset X$,
\item a sequence of integers $\{n_i\}_{i=0}^\infty$ with $n_0=0$ and $n_i \geq k_{g}(\eps_i)$.
\end{enumerate}
Then, setting
$$
l_i=\sum_{s=0}^{i-1}n_s\,\,\text{for }i\in\Zp,
$$
we can find a point $z\in X$ such that for every $j\in\Zp$ we have
$$
f^{l_j}(z)\in \overline{B_{n_j}(g;x_j,\eps_j)}.
$$
\end{lem}
\begin{proof}
Since each number $k_{g}(\eps)$ in the definition of the almost specification property depends only on $\eps$, but not on $m$, for each $m\in\Zp$ we can find a point $z_m$
such that for every $j=1,\ldots,m$ we have
$$
f^{l_j}(z_m)\in B_{n_j}(g;x_j,\eps_j).
$$

We will use a diagonalization procedure to obtain a ``good'' subsequence of $\set{z_i}_{i=1}^\infty$, which, by an abuse of notation, we will also denote by $\set{z_i}_{i=1}^\infty$.
Since each $I(g; n_j, \eps_j)$ is finite, passing to a subsequence in $\set{z_i}_{i=1}^\infty$ we can assume that there is $\Lambda_1\in I(g; n_1, \eps_1)$ such that $f^{l_1}(z_j)\in B_{\Lambda_1}(x_1,\eps_1)$ for every $j\geq 1$. Similarly, passing again to a subsequence in $\set{z_i}_{i=2}^\infty$ if necessary (i.e. preserving $\Lambda_1$ and $z_1$), we can find $\Lambda_2\in I(g; n_2, \eps_2)$ such that
$f^{l_1}(z_j)\in B_{\Lambda_1}(x_1,\eps_1)$ for every $j\geq 1$ and $f^{l_2}(z_j)\in B_{\Lambda_2}(x_2,\eps_2)$ for every $j\geq 2$.
Continuing in this way, for each $i\in\Zp$ we can construct sets
$\Lambda_i\in I(g; n_i, \eps_i)$ such that $f^{l_i}(z_j)\in B_{\Lambda_i}(x_i,\eps_i)$ for every $j\geq i$.

Passing one more time to a subsequence we may assume that the limit $z=\lim_{m\to \infty}z_m$ exists.
But then for each $i\in\Zp$, every $j\geq i$, and every $k\in \Lambda_i$ we have $\dd(f^k(x_i), f^k(f^{l_i}(z_j)))<\eps$ and hence
$\dd(f^k(x_i), f^k(f^{l_i}(z)))\leq \eps$. In other words, for $j\in\Zp$ we have
$$
f^{l_j}(z)\in \overline{B_{\Lambda_j}(x_j,\eps_j)}\subset \overline{B_{n_j}(g;x_j,\eps_j)}.
$$
\end{proof}

Roughly speaking, the above lemma says that given an infinite sequence of triples consisting of a point, a large enough length and a precision, we
can find a point (taken from an orbit of some point $z$) tracing an orbit of each given point for a given length with the given precision (with possible mistakes, but their number is bounded above by a mistake function $g$). That point is $f^{l_j}(z)$, and it is tracing the triple $(x_j, m_j,\eps_j)$.

%%%%%%%%%%%%%%%%%%%%%%%%%%%%%%% THM:ALMOST-SPEC-IMPLIES-AASP %%%%%%%%%%%%%%%%%%%%%%%%%%%%%%%%%%%%%%%%%%%%%%%%%%%%%%%%%%
\begin{thm}\label{thm:almost-spec-implies-aasp}
If a surjective compact dynamical system $(X,f)$ has the almost specification property, then it has the asymptotic average shadowing property.
\end{thm}

\begin{proof} Let $\set{x_j}_{j=0}^\infty$ be an asymptotic average pseudo-orbit of $f$. By Lemma \ref{lem:surj-almost-spec-implies-chain-mix} $f$ is chain mixing, and therefore we may use Lemma \ref{lem:apo-from-aapo} to obtain an asymptotic pseudo-orbit of $f$, denoted $\set{y_j}_{j=0}^\infty$,  such that $d(\set{j : x_j\neq y_j})=0$. It is enough to show that
$\set{y_j}_{j=0}^\infty$ can be asymptotically shadowed on average by some point $z\in X$.

Let $g$ be a mistake function for $f$. For each $k\geq 1$ we take an integer $n_k$ such that $n_k>k_g(1/2^k)$ and
$n_k> 2^k g(n_k,1/2^k)$. We may assume $n_k< n_{k+1}$. Using compactness of $X$ and continuity of $f$, for each $k$ we can also find a constant $\delta_k>0$ such that every $\delta_k$-chain of length $n_k$ is $1/2^k$-shadowed by
its first point. Since $\set{y_j}_{j=0}^\infty$ is an asymptotic pseudo-orbit we can find a sequence $\set{m_k}_{k=1}^\infty$ such that for each $k\in\Zp$ the sequence $\set{y_j}_{j=m_k}^\infty$ is a $\delta_k$-pseudo-orbit. Clearly we can choose $\set{m_k}_{k=1}^\infty$ so that for each $k\in\Zp$ we will have in addition:
\begin{enumerate}
\item $m_{k+1}>4^km_k$,
\item $n_k$ divides $(m_{k+1}-m_k)$, equivalently, there is $s_k>0$ is such that $m_{k+1}-m_k=s_k n_k$.
\item $4^k n_{k+1}<m_{k+1}$.
\end{enumerate}
We call any point $y_{m_k + s n_k}$, where $k\in\Zp$ and $0\leq s<s_k$ an \emph{initial point of order $k$}. Note that by our choice of parameters, for every $k\in\Zp$ and every $0\leq s<s_k$ a sequence $\set{y_j}_{j=m_k+sn_k}^{m_k+(s+1)n_{k}-1}$ is $1/2^k$-shadowed by its first (initial) point $y_{m_k + s n_k}$.
Set $\eps_0=1$, $k_0=m_1$ and let $\set{\eps_j}_{j=1}^\infty$ ($\set{k_j}_{j=1}^\infty$, respectively) be a non-increasing (non-decreasing, respectively) sequence such that $1/2^k$, ($n_k$, respectively) repeats $s_k$ times as an element of the sequence. We may plug this data (sequence of all initial points and sequences $\eps=\set{\eps_j}_{j=0}^\infty$, $k=\set{k_j}_{j=0}^\infty$) into  Lemma~\ref{lem:inf_almostspec} to obtain a point $z\in X$ such that for any $k\in\Zp$ and every $0\leq s<s_k$ we have $f^{m_k+sn_k}(z)\in \overline{B_{n_k}(g;y_{m_k + s n_k},1/2^k)}$.
The corresponding set of mistakes is defined as the set $\set{m_k+sn_k\le j<m_k+(s+1)n_k:\dd(f^{j}(z),f^{j-m_k-sn_k})\ge 1/2^k}$.
\begin{figure}
\begin{align*}
{\eps}&=(1,\underbrace{1/2,\ldots,1/2}_{\text{$s_1$ times}},\underbrace{1/2^2,\ldots,1/2^2}_{\text{$s_2$ times}},\ldots,\underbrace{1/2^k,\ldots,1/2^k}_{\text{$s_k$ times}},\ldots),\\
{x}&= (y_0,\underbrace{y_{m_1},y_{m_1+n_1},y_{m_1+2n_1},\ldots,  y_{m_2-n_1}}_{\text{$s_1$ initial points of order $1$}},\ldots,\underbrace{y_{m_k},y_{m_k+n_k},y_{m_k+2n_k},\ldots,  y_{m_k-n_k}}_{\text{$s_k$ initial points of order $k$}},\ldots),\\
{k}&=(m_1,\underbrace{n_1,\ldots,n_1}_{\text{$s_1$ times}},\underbrace{n_2,\ldots,n_2}_{\text{$s_2$ times}},\ldots,\underbrace{n_k,\ldots,n_k}_{\text{$s_k$ times}},\ldots).
\end{align*}
\caption{Sequences $\eps$, ${x}$, $k$ to which Lemma~\ref{lem:inf_almostspec} is applied in the proof of Theorem \ref{thm:almost-spec-implies-aasp}.}
\end{figure}

Let $J\subset \N$ be the complement of the union of all sets of mistakes. Note that if $j\in J\cap [m_k,m_{k+1})$ then $\dd(f^j(z),y_j)<1/2^k+1/2^k$, because
$y_j$ is $1/2^k$ traced by the orbit of the closest preceding initial point and $j\in J$ means that the orbit of $z$ traces without a mistake at $j$.
Therefore we have $\lim_{j\in J}d(f^j(z),y_j)=0$.
But by the choice of $z$ we also have that
\begin{eqnarray*}
\frac{1}{m_{k+1}}\# (J \cap [m_k,m_{k+1}))
&\geq & \frac{1}{m_{k+1}} \cdot \frac{(m_{k+1}-m_k)}{n_k} (n_k - g(n_k,1/2^k))\\
&= & (1-m_k/m_{k+1}) ( 1- g(n_k,1/2^k)/n_k )\\
&\geq & (1-1/4^k)(1-1/2^k).
\end{eqnarray*}
Next, if we fix $n\in [m_{k+1}, m_{k+2})$, where $n=m_{k+1}+sn_{k+1}+r$, $0\leq r<n_{k+1}$, then we also have
\begin{eqnarray*}
\frac{1}{n}\density(J\mid n)&\geq&\frac{1}{m_{k+1}}\cdot\frac{m_{k+1}}{n}\# (J \cap [m_k,m_{k+1}))+ \frac{1}{n}\# (J \cap [m_{k+1},n))\\
&\geq & (1-1/4^k)(1-1/2^k) \frac{m_{k+1}}{n}+ \frac{s (n_{k+1}-g(n_{k+1},1/2^{k+1}))}{n}\\
&\geq & (1-1/4^k)(1-1/2^k) \frac{m_{k+1}}{n} + \frac{(1-1/2^{k+1})s n_{k+1}}{n}\\
&\geq & (1-1/4^k)(1-1/2^k) \frac{m_{k+1}}{n} + \frac{(1-1/4^k)(1-1/2^k)s n_{k+1}}{n}\\
&\geq & (1-1/4^k)(1-1/2^k)(1 - n_{k+1}/m_{k+1})\\
&\geq& (1-1/4^k)(1-1/2^k)(1-1/4^k) \xrightarrow[k\ra\infty]{} 1.
\end{eqnarray*}
We conclude that $d(J)=1$, which by Lemma \ref{Walters} ends the proof.
\end{proof}

\subsection{Asymptotic average shadowing implies average shadowing}

In this section we show that the asymptotic average shadowing imply average shadowing.
Next theorem shows that, under the assumption of chain mixing, the average shadowing property is all about
average shadowing of pseudo-orbits.
\begin{thm}\label{thm:normal_orb_asp}
Assume that a compact dynamical system $(X,f)$ is chain mixing. Then the following conditions are equivalent:
\begin{enumerate}
\item\label{thm:normal_orb_asp:1} $f$ has the average shadowing property,
\item\label{thm:normal_orb_asp:2} for every $\eps>0$ there is $\delta>0$ such that every $\delta$-pseudo-orbit $\set{x_i}_{i=0}^\infty$
is $\eps$-shadowed in average by some point $x\in X$.
\end{enumerate}
\end{thm}

\begin{proof}
Since every $\delta$-pseudo-orbit is also a $\delta$-average pseudo-orbit we only need to prove $\eqref{thm:normal_orb_asp:2}\Rightarrow\eqref{thm:normal_orb_asp:1}$.

Fix any $\eps>0$ and let $\gamma$ be provided to $\eps/5$ such that every $\gamma$-pseudo-orbit
is $\eps/5$-shadowed in average by some point in $X$. By chain mixing and compactness of $X$ there is $M$ such that for every
$x,y\in X$ there is a $\gamma$-pseudo-orbit of length $M$ from $x$ to $y$.

Put $\delta=\eps \gamma/(3M\diam(X))$, fix any $\delta$-average-pseudo-orbit $\set{x_i}_{i=0}^\infty$, and let $N$ be the constant for it
from the definition of a $\delta$-average-pseudo-orbit. Without loss of generality we may assume that $M<N$.

We will now construct a sequence $\set{y_i}_{i=0}^\infty$ which is a genuine pseudo-orbit and shadows $\set{x_i}_{i=0}^\infty$ on average.
Starting with $j=0$ and going up through all the natural numbers we initially put (some of these blocks will be modified later):
$$
y_{jN},y_{jN+1},\ldots,y_{jN+N}=x_{jN},x_{jN+1},\ldots,x_{jN+N}.
$$
Next, if there is $jN\leq i<jN+N$ such that $\dd(f(y_i),y_{i+1})\geq \gamma$, we choose $k$ so that $jN\leq k\leq i< k+M-1\leq jN+N$ and replace $y_{k},y_{k+1},\ldots,y_{k+M-1}$ with any $\gamma$-pseudo-orbit of length $M$ from $y_{k}$ to $y_{k+M-1}$. We repeat these replacements until $y_{jN},y_{jN+1},\ldots,y_{jN+N}$ becomes a $\gamma$-pseudo-orbit. Note that no replacement changes $y_{jN}$ or $y_{jN+N}$ and therefore the resulting infinite sequence $\set{y_i}_{i=0}^\infty$ is a $\gamma$-pseudo-orbit.

Observe that for all $j\in\N$ we have
$$
M\cdot\#\set{i \in [jN,jN+N) \colon\dd(f(x_i),x_{i+1})\geq \gamma}\geq\#\set{i \in [jN,jN+N) \colon x_i\neq y_i)},
$$
and therefore
$$
\delta M >\frac{M}{N}\sum_{i=0}^{N-1}\dd(f(x_{jN+i}),x_{jN+i+1})\geq\frac{M}{N}\gamma\cdot\#\set{i \in [jN,jN+N) \colon \dd(f(x_i),x_{i+1})\geq \gamma}.
$$
It follows that
\begin{align*}
\frac{\diam(X)}{N}\cdot\#\set{i \in [jN,jN+N)\colon x_i\neq y_i}&\leq \frac{\delta M\diam (X)}{\gamma}.
\end{align*}
But by the definition of $\delta$ we have
\[
\frac{\delta M\diam (X)}{\gamma} = \frac{M\diam(X)}{\gamma}\cdot\frac{\eps\gamma}{3M\diam (X)}\leq  \frac{\eps}{3}.
\]

We are now ready for the final calculation. Let $z$ be a point that $\frac{\eps}{5}$-shadows in average the $\gamma$-pseudo-orbit $\set{y_i}_{i=0}^\infty$.
There is $K\in\N$ such that for every $n>K$ we have
$$
\frac{1}{n}\sum_{i=0}^n d(f^i(z),y_i)<\eps/4.
$$
We may also assume that $N/K<\frac{\eps}{4}$. Fix any $n>K$ and let $s,l\geq 0$ be such that $n=sN+l$, where $l<N$.
Let $A=\set{i<n : x_i= y_i}$ and $B=\set{i<n : x_i\neq y_i}$.
Then
\begin{eqnarray*}
\frac{1}{n}\sum_{i=0}^n d(f^i(z),x_i)&= & \frac{1}{n}\sum_{i\in A} d(f^i(z),y_i)+\frac{1}{n}\sum_{i\in B} d(f^i(z),x_i)\\
&\leq &\frac{1}{n}\sum_{i=0}^n d(f^i(z),y_i)+\frac{l}{n}+ \frac{1}{n}\sum_{j=0}^{s-1}\sum_{i\in B \cap [Nj,Nj+N)}\diam (X)\\
&\leq &\frac{\eps}{4}+\frac{\eps}{4}+\frac{N}{n}\sum_{j=0}^{s-1} \frac{\diam X}{N}\cdot \#(B \cap [Nj,Nj+N))\\
&\leq & \frac{\eps}{2}+\frac{sN}{n}\frac{\eps}{3} \leq \frac{5\eps}{6}.
\end{eqnarray*}
This implies that $\limsup_{n\to \infty }\frac{1}{n}\sum_{i=0}^n d(f^i(z),x_i)\leq \frac{5\eps}{6} <\eps$ and the proof is completed.
\end{proof}

\begin{thm}\label{thm:aasp-implies-asp}
If $f\colon X\to X$ is a surjection and the compact dynamical system $(X,f)$ has the asymptotic average shadowing property, then it also has the average shadowing property.
\end{thm}
\begin{proof}
By \cite[Theorem~3.1]{KuOp2} every surjective dynamical system with the asymptotic average shadowing property is chain mixing.
Assume on the contrary that $(X,f)$ does not have the average shadowing property. Then by Theorem~\ref{thm:normal_orb_asp},
there is an $\eps>0$ such that for every $n\in\Zp$ there is a $\frac{1}{n}$-pseudo-orbit $\{\alpha^{(n)}_j\}_{j=0}^\infty$ which is not
$\eps$-shadowed in average by any point in $X$. Then by compactness of $X$ for every $n\in\Zp$ there exist $k(n)\in\Zp$ and an initial block $\gamma^{(n)}=\set{\gamma^{(n)}_i}_{i=0}^{k(n)-1}=\set{\alpha^{(n)}_0,\alpha^{(n)}_1,\ldots, \alpha^{(n)}_{k(n)-1}}$
of $\alpha^{(n)}$ such that for every $z\in X$ we have $\frac{1}{k(n)}\sum_{i=0}^{k(n)-1} \dd(f^i(z),\alpha^{(n)}_i)\geq \eps$. Note that $\lim_{n\to \infty}k(n)=\infty$, as otherwise by the continuity of $f$ for a sufficiently large $n$ the pseudo-orbit $\gamma^{(n)}$ would be $\eps$-shadowed by its initial point $\alpha^{(n)}_0$.

By compactness of $X$ and chain mixing of $f$, for every $n\in\Zp$ there is $m(n)\in\N$ such that any two points can be connected by a $\frac{1}{n}$-pseudo-orbit of length $m(n)$. For $a<b$ any $\frac{1}{b}$-pseudo-orbit is also $\frac{1}{a}$-pseudo-orbit and the sequence $m(n)$ is fixed already, so, without loss of generality, going to a subsequence in $\{k(n)\}_{n=1}^\infty$ if necessary, we may assume that for every $n\in\Zp$ we have $k(n+1)> 2 \sum_{i=1}^n (k(i)+m(i))$. For every $n\in\Zp$ denote by $\eta^{(n)}$ a $\frac{1}{n}$-pseudo-orbit of length $m(n)-2$ such that
$\gamma^{(n)}_{k(n)-1},\eta^{(n)},\gamma^{(n+1)}_{1}$ is also $\frac{1}{n}$-pseudo-orbit (such $\eta^{(n)}$ exists by the choice of $m(n)$).

Denote by $\zeta$ the following concatenation of pseudo-orbits
$$
\zeta=\set{\zeta_i}_{i=0}^\infty=\gamma^{(1)}\eta^{(1)}\gamma^{(2)}\eta^{(2)}\ldots \gamma^{(n)}\eta^{(n)}\ldots
$$
and observe that $\zeta$ is an asymptotic pseudo-orbit. Let $z\in X$ be a point provided by the asymptotic average shadowing property that asymptotically shadows $\zeta$ in average.
Denote by $r(n)$ the number of elements in the pseudo-orbit $\gamma^{(1)}\eta^{(1)}\ldots\eta^{(n-1)}\gamma^{(n)}$.
There exists $N\in\Zp$ such that
$$
\frac{1}{r(N)}\sum_{i=0}^{r(N)-1}\dd(f^i(z),\zeta_i)<\eps/2.
$$
Note that
$$r(N)\leq k(N)+\sum_{i=1}^{N-1} (k(i)+m(i))\leq (1+\frac{1}{2})k(N)\leq 2k(N),$$
and so
\begin{eqnarray*}
\frac{1}{k(N)}\sum_{i=0}^{k(N)-1}\dd(f^{r(N)-k(N)+i}(z),\gamma^{(N)}_i)&=&\frac{1}{k(N)}\sum_{i=r(N)-k(N)}^{r(N)-1} \dd(f^i(z),\zeta_i)\\
&\leq& 2 \frac{1}{r(N)}\sum_{i=r(N)-k(N)}^{r(N)-1}\dd(f^i(z),\zeta_i) < \eps,
\end{eqnarray*}
which contradicts the choice of ${\gamma}^{(N)}$, and completes the proof.
\end{proof}

\subsection{Under shadowing}

Combining our new results with previous research we obtain the following:

\begin{thm}\label{thm:summary}
Let $X$ be a compact metric space. If $f\colon X\ra X$ is a continuous map with the shadowing property, then the following conditions are equivalent:
\begin{enumerate}
\item\label{summary:tt} $f$ is totally transitive,
\item\label{summary:wm} $f$ is topologically weakly mixing,
\item\label{summary:mix} $f$ is topologically mixing,
\item\label{summary:spec} $f$ is surjective and has the specification property,
\item\label{summary:almost-spec} $f$ is surjective and has the almost specification property,
\item\label{summary:aasp} $f$ is surjective and has the asymptotic average shadowing property.
\item\label{summary:asp} $f$ is surjective and has the average shadowing property,
\end{enumerate}
Moreover, if $f$ is c-expansive, then any of the above conditions is equivalent to the periodic specification property of $f$.
\end{thm}

\begin{proof}Implications $\eqref{summary:almost-spec} \Longrightarrow \eqref{summary:aasp}\Longrightarrow\eqref{summary:asp}$ were proved in Theorems~\ref{thm:almost-spec-implies-aasp} and \ref{thm:aasp-implies-asp}, respectively.
To see $\eqref{summary:asp}\Longrightarrow\eqref{summary:tt}$, first observe that $f$ is chain transitive by \cite[Theorem~3.4]{KorASP}.
It is also known that if $f$ has average shadowing property then so does $f^n$ for every $n\geq 1$ (e.g. see Lemma~3.3 in \cite{Niu}).
In particular, by shadowing, $f$ is totally transitive.

Equivalence of conditions \eqref{summary:tt}--\eqref{summary:spec} as well as the ``moreover'' part are \cite[Theorem~1]{KwOp}.
The proof that specification implies almost specification can be found in \cite[Proposition 2.1]{PS} (see also comments in \cite{T}).
\end{proof}

%%%%%%%%%%%%%%%%%%%%%%%%%%%%%%%%%%%%%%%%%%%%%%%%%%%%%%%%%%%%%%%%%%%%%%%%%%%%%%%%%%%%%%%%%%%%%%%%%%%%%%%%%%%%%%%%%%%%%
\section{Weak mixing}

In general, no recurrence property is implied by almost specification or the average shadowing property.
We show this in Section \ref{sec:examples} below.
Nevertheless, under some additional assumptions one can get a positive result. For example,
Niu \cite[Corollary 3.8]{Niu} proved that a system with the average shadowing property
and a dense set of minimal points must be weakly mixing.
Here we prove that a compact system with almost specification and an invariant measure with full support
has a dense set of minimal points. We do not know whether the same holds if we assume only (asymptotic) average shadowing
(Question \ref{q:dense_minimal}). However, we can generalize Niu's result and show that his conclusion
holds under (a priori weaker) assumption of the existence of an invariant measure with full support.
But if the answer to Question \ref{q:dense_minimal} or Question \ref{q:converse_implications} is positive,
then Theorem \ref{thm:measure+asp-imply-wm} is merely a reproof of Niu's result. Note that there are examples of non-trivial proximal systems with
an invariant measure of full support. For example one may consider the restriction of a shift constructed in \cite[Theorem 5.6]{KwSpacing}
to the support of any of its invariant measures of positive entropy (there is at least one such measure).

\begin{thm}\label{thm:dense_minimal}
Assume that a compact dynamical system $(X,f)$ has an invariant measure with full support. If $f$ has the almost specification property, then it has a dense set of minimal points.
\end{thm}
\begin{proof}
Fix any nonempty open set $U$. We can find $\eps>0$ and nonempty open sets $W\subset V\subset U$ such that the $\eps$-neighborhood of $\overline{W}$ is contained in $V$ and $\overline{V}\subset U$. Since $U$ is not universally null we conclude from Theorem \ref{thm:xi_univnull} and Lemma \ref{lem:density} that we can find $\gamma>0$ and a point $x\in W$ such that $d(N(x,W))=\gamma$. Let $N\in \N$ be such that for every $n\ge N$ we have
$\density(N(x,W)|n)\ge n\gamma/2$. Using the almost specification property we can find an $M>0$ such that for all $m\ge M$ we have $g(\eps,m)< m\gamma/2$. Let $n=\max\{N,M,k_g(\eps)\}$. Let $\{x_j\}_{j=0}^\infty$, $\{\eps_j\}_{j=0}^\infty$, $\{n_j\}_{j=0}^\infty$  be constant sequences, where $x_j=x$, $\eps_j=\eps$, $n_j=n$ for every $j\in\N$. By Lemma \ref{lem:inf_almostspec} there is a point $y$ such that $f^{jn}(y)\in B_n(g;x,\eps)$ for every $j\in\N$. We claim that
$N(y,V)$ is syndetic (has gaps bounded by $2n$). Assume conversely that $f^l(y)\notin V$ for $2n$ consecutive indices $l$. In particular,
for some $j\ge 0$ and every $0\le i < n$ such that $f^i(x)\in W$ we necessarily have $\dd(f^{jn+i}(z),f^i(x))\ge\eps$. But this leads to a contradiction: $n\gamma/2\le g(\eps,n)< n\gamma/2$. Therefore $N(y,V)$ is syndetic with gaps bounded by $2n$ as claimed. By the Auslander-Ellis Theorem \cite[Theorem 8.7]{Fur} there is a minimal point $z$ proximal to $y$. It is easy to see that $f^l(z)\in U$ for some $l\ge 0$. Since $f^l(z)$ is also a minimal point the proof is finished.
\end{proof}

%We need the following auxiliary result:

\begin{lem}\label{lem:asp_mes}
Let $p,q\in X$. Let $A,B\subset \N$ be such that $d^*(A)>0$ and $d^*(B)>0$. If $X$ is compact and $f\colon X\ra X$ has the average shadowing property, then for every $\eps>0$ there exist a point $z\in X$ and integers $a_0<b_0<a_1<b_1<\ldots$ such that for every $i\in\N$ we have $a_i\in A$, $b_i\in B$, $\dd(f^{a_i}(z),f^{a_i}(p))<\eps$, and $\dd(f^{b_i}(z),f^{b_i}(q))<\eps$.
\end{lem}

\begin{proof}
We may assume without loss of generality that $\mbox{diam}(X)=1$. Let $0<\gamma<\min(d^*(A), d^*(B))$. Fix any $\eps>0$ and let $\alpha>0$ be given for $\gamma\eps$ by the average shadowing property. Using the definition of the upper asymptotic density we can easily construct an increasing sequence of integers $0=n_0<n_1<\ldots$ such that $2<\alpha n_1$ and for any $i\in\Zp$ we have:
\begin{enumerate}
\item $2 n_i<n_{i+1}$,
\item $\frac{1}{n_{2i+1}}\density(A\cap[n_{2i},n_{2i+1})\mid n_{2i+1})>\gamma$,
\item $\frac{1}{n_{2i}}\density(B\cap[n_{2i-1},n_{2i})\mid n_{2i})>\gamma$.
\end{enumerate}
Next we define a sequence of points $y_j\in X$ by
\[y_j=
\begin{cases}
f^j(p) &\text{if }j\in [n_{2i}, n_{2i+1}) \text{ for some }i,\\
f^j(q) & \text{otherwise}.
\end{cases}\]
We claim that the sequence $\set{y_j}_{j=0}^\infty$ is an $\alpha$-average-pseudo-orbit of $f$. To prove it observe that $\dd(f(y_{j-1}), y_j)$ can be positive only if $j\in\set{n_i}_{i=0}^\infty$. But if we put $N=n_1$ then, by the definition of the sequence $\set{n_j}_{i=0}^\infty$, for every $j\in\N$ the interval $[j,j+N]$ can contain at most one element of this sequence. If we now fix any $n\geq N$ we may find $s>0$ such that $Ns\leq n <(s+1)N$. Then for every $k\geq 0$ we have \[\frac{1}{n}\sum_{i=0}^{n-1}\dd(f(y_{i+k}),y_{i+k+1})<\frac{s+1}{sN}\leq \frac{2}{n_1}<\alpha,\] proving that $\set{y_j}_{j=0}^\infty$ is an $\alpha$-average-pseudo-orbit. It is, therefore, $\gamma\eps$-shadowed in average by some $z\in X$.

If the choice of sequences $\{a_i\}_{i=0}^\infty$ and $\{b_i\}_{i=0}^\infty$ is not possible, then, for all $i$ large enough, at least one of the following two conditions must be satisfied:
\begin{enumerate}
\item for all $j\in A\cap[n_{2i},n_{2i+1})$ we have $\dd(f^j(z),f^j(p))\geq\eps$,
\item for all $j\in B\cap[n_{2i-1},n_{2i})$ we have $\dd(f^j(z),f^j(q))\geq\eps$.
\end{enumerate}
We may assume without loss of generality that the first condition holds. Then for all $i$ large enough we have
\[
\frac{1}{n_{2i+1}}\sum_{s=0}^{n_{2i+1}-1}\dd(f^s(z),f^s(p))\geq\eps \frac{1}{n_{2i+1}}\density(A\cap[n_{2i},n_{2i+1})\mid n_{2i+1})>\gamma\eps,
\]
contradicting that $z$ $\gamma\eps$-shadows on average an $\alpha$-average pseudo-orbit $\{y_j\}_{j=0}^\infty$.
\end{proof}

The following theorem generalizes \cite[Theorem 3.7]{Niu}.

\begin{thm}\label{thm:measure+asp-imply-wm}
Assume that a compact dynamical system $(X,f)$ has an invariant measure with full support. Then if $f$ has the average shadowing property, then it is weakly mixing.
\end{thm}

\begin{proof}
Fix any nonempty open sets $U,V$. We can find $\eps>0$ and nonempty open sets $U',V'$ such that the $\eps$-neighborhood of $\overline{U'}$ is contained in $U$ and the $\eps$-neighborhood of $\overline{V'}$ is contained in $V$. Since neither $U'$, nor $V'$ is universally null we conclude from Theorem \ref{thm:xi_univnull} that we can find $\gamma>0$, points $p\in U'$, $q\in V'$, and sets $A,B\subset \N$ such that $d^*(A)>\gamma$, $d^*(B)> \gamma$, and $f^i(p)\in U'$, $f^j(q)\in V'$ for every $i\in A$, $j\in B$. Then by Lemma~\ref{lem:asp_mes} there is $z\in X$ and $i<j$ such that $d(f^i(z),f^i(p))<\eps$ and $d(f^j(z),f^j(q))<\eps$. In particular, $f^i(z)\in U$ and $f^j(z)\in V$. This shows that $f$ is transitive.

It was shown in Proposition 3.5 in \cite{Niu} that if $f$ has the average shadowing property, then $f\times f$ also has this property. Obviously $\mu \times \mu$ is a fully supported invariant measure for $f\times f$, so by the above arguments, $f\times f$ is transitive, which ends the proof.
\end{proof}

%%%%%%%%%%%%%%%%%%%%%%%%%%%%%%%%%%%%%%%%%%%%%%%%%%%%%%%%%%%%%%%%%%%%%%%%%%%%%%%%%%%%%%%%%%%%%%%%%%%%%%%%%%%%%%%%%%%%%
\section{Measure center}\label{sec:sec_measure_center}

In this section we prove that the generalized shadowing and specification are connected with the behaviour of the dynamical system on its measure center.

\subsection{Almost specification}
In what follows, we will use the notation introduced in Section~\ref{prelims}, in particular the definition of the function $\xi$.

%{\color{red}
%\begin{cor}\label{cor:mistake-measure-center}
%Assume that a closed set $A$ contains the measure center of a compact dynamical system $(X,f)$. Then for every $\eps>0$ the function
%\[
%\gamma'_A(n,\eps)=\max_{x\in X}
%\#%\big|
%\set{ 0\leq j < n : \dd(f^j(x),A)> \eps }
%%\big|
%\]
%is a mistake function (i.e. $\gamma'_A(n,\eps) \leq  \gamma'_A(n+1,\eps)$ and $\lim_{n\to \infty}\frac{\gamma'_A(n,\eps)}{n}= 0$).
%\end{cor}
%\begin{proof} The case $A=X$ is trivial, so we assume that $A\neq X$.
%Note that for a sufficiently small $\eps>0$ the set $U=\set{x : \dd(x,A)> \eps }$ is open and nonempty. Moreover $\gamma'_A(n,\eps)=\eta(U,n)$, so
%\[\lim_{n\to \infty}
%\frac{\gamma'_A(n,\eps)}{n}=\xi(U).\]
%But $\xi(U)=0$, as otherwise by Lemma \ref{lem:density} there would exist a point $x$ with $d(N(x,U))>0$, which would imply $\mu(U)>0$ for some invariant measure, contradicting that $A$ contains the measure center.
%\end{proof}
%}

\begin{thm}\label{thm:almostspec_measure}
If a closed invariant set $A$ contains the measure center of a compact dynamical system $(X,f)$ and $f|_A$ has the almost specification property on $A$, then so does $f$ on $X$.
\end{thm}

\begin{proof}
Let $g_A\colon \Zp\times (0,\eps_0]\mapsto\mathbb{N}$ be a mistake function for $f|_A$ and let $k_{g_A}$ be provided by the almost specification property.
We claim that it is sufficient to prove that there is a mistake function $\gamma\colon \Zp\times (0,\eps_0]\ra\mathbb{N}$
fulfilling
\begin{eqnarray*}
(\star)&&\mbox{for every $\eps\in (0,\eps_0]$ there is $N=N_\gamma(\eps)$ such that for every $x\in X$ and every $n\geq N$}\\
&&\mbox{there is a point $z\in A\cap B_n(\gamma;x,\eps)$.}
\end{eqnarray*}

First we prove that if there is $\gamma$ satisfying ($\star$), then $f$ has the almost specification property on $X$ with the mistake function
$$
g(n,\eps)=g_A(n,\eps/2)+\gamma(n,\eps/2), \qquad \text{where }(n,\eps)\in \Zp\times (0,\eps_0].
$$
and with a function $k_g(\eps)=\max\{k_{g_A}(\eps),N_\gamma(\eps)\}$. In order to prove the claim, fix any
$m\geq 1$, any $\eps_1,\ldots,\eps_m > 0$, any points $x_1, \ldots, x_m \in X$, and any integers
$n_1 \geq k_{g}(\eps_1),\ldots,n_m \geq k_{g}(\eps_m)$. Put $n_0=0$ and denote
$$
l_j=\sum_{s=0}^{j-1}n_s,\,\text{for }j=1,\ldots,m.
$$
By the choice of $\gamma$, for each $j=1,\ldots,m$ there is a set $\Gamma_j\in I(\gamma;n_j,\eps_j/2)$ and a point $z_j\in A\cap B_{\Gamma_j}(x_j,\eps_j/2)$.
Let $z$ be a point obtained by the almost specification property for the aforementioned points $z_1,\ldots,z_m$ and constants $n_j$ and $\eps_j/2$.
Strictly speaking, point $z\in X$ is such that for every $j=1,\ldots,m$ we have
$$
f^{l_j}(z)\in B_{n_j}(g_A;z_j,\eps_j/2).
$$
In particular, for every $j$ there is a set $\Lambda_j\in I(g_A;n_j,\eps_j/2)$ such that $f^{l_j}(z)\in B_{\Lambda_j}(g_A;y_j,\eps_j/2).$
Put $K_j=\Gamma_j\cap \Lambda_j$ and observe that
$$
n_j-\#K_j \leq (n_j-\#\Gamma_j) + (n_j-\#\Lambda_j)\leq \gamma(n_j,\eps_j/2)+g_A(n_j,\eps_j/2)=g(n_j,\eps_j)
$$
which shows that $K_j\in I(g;n_j,\eps_j)$. Additionally, if $i\in K_j$, then
$\dd(f^{l_j+i}(z),f^i(z_j))<\eps_j/2$ and $\dd(f^i(x_j),f^i(z_j))<\eps_j/2$, hence
$\dd(f^{l_j+i}(z),f^i(x_j))<\eps_j$. This proves that
$$
f^{l_j}(z)\in B_{K_j}(x_j,\eps_j)\subset B_{n_j}(g;x_j,\eps_j),
$$
and therefore it only remains to prove that $\gamma$ can be constructed in such a way that $(\star)$ is satisfied.

Take any $\eps\in(0,\eps_0]$ and any $s\geq k_{g_A}(\eps/2)$. Fix a constant $\delta>0$ such that for every $x,y\in X$ with $\dd(x,y)<\delta$ we have $\dd(f^j(x),f^j(y))<\eps/2$ for $j=0,\ldots,s-1$. Denote $P(n)=\sup_{x\in X}\#\{0\leq i<n\:|\:\dd(f^i(x),A)\geq\delta\}$ and observe that, by Corollary \ref{cor:measure_cent}, we have $\lim_{n\ra\infty}P(n)/n=0$. Pick $N_\gamma^s(\eps)$ large enough for the inequality
$$
s(P(n)+1)\leq\frac{n\cdot g_A(\eps/2,s)}{s}
$$
to hold for every $n\geq N_\gamma^s(\eps)$. Let $n\geq N_\gamma^s(\eps)$ be given. Assume that $n=ps+r$, where $p,r\in\N$ and $r<s$. For $j=1,\ldots,p$ define $n_j=s$, $\eps_j=\delta$, and let $z_j\in A$ be any point such that $\dd(z_j,f^{(j-1)s}(x))$ realizes the distance between the closed set $A$ and a point $f^{(j-1)s}(x)$. Let $z\in A$ be a point obtained by the almost specification property of $f|_A$ for the points $z_1,\ldots,z_p$ and constants $n_j$ and $\eps_j$. Observe that
\begin{eqnarray*}
\#\{0\leq i<n\:|\:\dd(f^i(x),f^i(z))\geq\eps\}&\leq&p\cdot g_A(\eps/2,s)+r+sP(n)\\
&\leq&\frac{n\cdot g_A(\eps/2,s)}{s}+s(P(n)+1)\\
&\leq&\frac{2n\cdot g_A(\eps/2,s)}{s}.
\end{eqnarray*}
To sum up, if $s\geq k_{g_A}(\eps/2)$ and $n\geq N_\gamma^s(\eps)$, then the orbit of length $n$ of every point $x\in X$ can be $\eps$-shadowed by the orbit of length $n$ of some point $z\in A$ with at most $2n\cdot g_A(\eps/2,s)/s$ errors.

Fix a sequence $\{s_i\}_{i=1}^\infty$ such that $k_{g_A}(\eps/2)=s_1<s_2<\ldots$ and put $N_\gamma(\eps)=k_{g_A}(\eps/2)$. Finally observe that the function
$$
\alpha(n,\eps)=\min_{N_\gamma^{s_i}(\eps)\leq n}\frac{2n\cdot g_A(\eps/2,s_i)}{s_i}
$$
satisfies ($\star$) aside from the fact that it does not necessarily have to be increasing with respect to $n$. Therefore $\gamma(n,\eps)=\max_{N_\gamma(\eps)\leq k\leq n}\alpha(k,\eps)$ is a mistake function satisfying ($\star$).

\end{proof}

\subsection{Average and asymptotic average shadowing properties}
In \cite{KuOp2} the authors considered properties that are sufficient to extend asymptotic average shadowing property
from a closed invariant set $A$ to the whole space (Theorem~\ref{thm:almostspec_measure} is, in fact, motivated by these studies).
Recall that by Corollary~\ref{cor:measure_cent} if a set $A$ contains the measure center of $f$ then for every $\varepsilon >0$ there exists $N\in\N$ such that for every $x\in X$
and $n\ge N$
we have
$$
\frac{1}{n}\#\set{0\leq i<n \;:\; \dd(f^i(x),A) <\varepsilon} >1-\varepsilon.
$$
Combining this observation with results of \cite{KuOp2} we obtain the following.
\begin{thm}
If a closed invariant set $A\subset X$ contains the measure center of a compact dynamical system $(X,f)$ and $f|_A$ has the asymptotic average shadowing property on $A$, then so does $f$ on $X$.
\end{thm}
\begin{proof}
It follows from Corollary~\ref{cor:measure_cent} and \cite[Theorem~3.3]{KuOp2}.
\end{proof}

To complete our considerations we will also prove the similar theorem for the average shadowing property.
In the proof we will make use of Corollary~\ref{cor:measure_cent} together with \cite[Lemma~3.5]{KuOp2}, which combine to yield the following lemma.

\begin{lem}\label{lematB}
If $f\colon X\ra X$ is a continuous map of a compact metric space and $A\subset X$ is a closed invariant set
 containing the measure center of $f$, then
for every $\eta >0$ there exist $K\in\N$ and $\beta>0$ such that for every $k\geq K$ and for every $\{x_i\}_{i=0}^{k-1}$ which is a $\beta$-pseudo-orbit of $f$ we have
\begin{equation}
\frac{1}{k}\cdot\#\set{0\leq i<k : \dd(x_i,A) <\eta} >1-\eta.\label{eq:eps_pseudo}
\end{equation}
\end{lem}

This lemma can be applied to obtain the following property.
\begin{lem}\label{lematC}
Let $f\colon X\ra X$ be a continuous map of a compact metric space and let $A\subset X$ be a closed invariant set
 containing the measure center of $f$.
Then for every $\eps>0$ there is $0<\delta<\eps$ such that for every $\delta$-average-pseudo-orbit $\set{x_n}_{n=0}^\infty$ there is
an $\eps$-average-pseudo-orbit $\set{y_n}_{n=0}^\infty\subset A$ such that
$$
\limsup_{n\ra \infty}\frac{1}{n}\#\set{0\leq i< n\colon\dd(x_i,y_i)\geq \eps} <\eps.
$$
\end{lem}

\begin{proof}
Fix any $\eps>0$ and let $N\in\Zp$ be such that $3\diam(X)/N<\eps$. Use continuity of $f$ to obtain $\eta>0$ such that $\eta<\eps/2$ and every $\eta$-pseudo-orbit $z_0,\ldots,z_N$ is $\eps$-shadowed by any point $x\in X$ such that $\dd(x,z_0)<\eta$.
Let $K$ and $\beta$ be provided to $\eta$ by Lemma~\ref{lematB}. Without loss of generality we may assume that $N$ divides $K$.
Fix $\delta<\min\set{\eps\eta/2K,\beta,\eta}$ and any $\delta$-average pseudo-orbit $\set{x_i}_{i=0}^\infty$. Note that by the choice of $\beta$ the set $\set{j \colon \dd(x_j,A)<\eta}$ is infinite.

We are now ready to define $\set{y_i}_{i=0}^\infty$. We arbitrarily choose any $y_0\in A$, start with $i=0$, and perform infinitely the following procedure:
\begin{enumerate}
\item find the first $j>i$ such that $\dd(x_j,a)<\eta$ for some $a\in A$,
\item for all $i<p<j$ put $y_p=f^{p-i}(y_i)$,
\item for all $j\leq p<j+N$ put $y_p=f^{p-j}(a)$,
\item increase $i$ to $j+N-1$,
\item go back to step $(1)$.
\end{enumerate}

Observe that the sequence $\set{y_i}_{i=0}^\infty$ obtained in this way is a concatenation of fragments of orbits of points from $A$, and the fragments are (with the possible exception of the first one) of length at least $N$. It follows that for every $n\geq N$, where $n=sN+r$ for some $0\leq r<N$, and for every $k\geq 0$ we have
$$
\frac{1}{n}\sum_{i=0}^{n-1}\dd(f(y_{i+k}),y_{i+k+1})\leq\frac{(s+2)\diam(X)}{sN}\leq\frac{3\diam(X)}{N}<\eps,
$$
and therefore $\set{y_i}_{i=0}^\infty$ is an $\eps$-average pseudo-orbit in $A$.

Additionally, by the definition of $\eta$, the first $N$ points of every fragment (again with the possible exception of the first fragment) $\eps$-shadow the $N$ points in $\set{x_i}_{i=0}^\infty$ with corresponding indices. Let $l$ denote the length of the first fragment of $\set{y_i}_{i=0}^\infty$. It follows that for $n$ large enough we have
\begin{eqnarray*}
\frac{1}{n}\cdot\#\set{0\leq i< n\colon\dd(x_i,y_i)\geq \eps}&\leq&\frac{1}{n}(l+n-\#\set{0\leq i< n\colon\dd(x_i,A)<\eta})\\
&\leq& \frac{l}{n}+1-(1-\eta)<\frac{\eps}{4}+\frac{\eps}{2}=\frac{3}{4}\eps
\end{eqnarray*}
and the assertion follows.
\end{proof}

\begin{thm}\label{glowne1}
If a closed invariant set $A\subset X$ contains the measure center of a compact dynamical system $(X,f)$ and $f|_A$ has the average shadowing property on $A$, then so does $f$ on $X$.
\end{thm}

\begin{proof}%[Proof of Theorem \ref{glowne1}.]
Without loss of generality we may assume that $\mbox{diam}(X)=1$. Fix any $\eps>0$. Let $\kappa<\eps/4$ be such that every $\kappa$-average pseudo-orbit in $A$ is $\eps/4$-shadowed in average by some point $x\in A$. Let $\delta$ be provided by Lemma~\ref{lematC} for $\kappa$. Take any $\delta$-average-pseudo-orbit $\set{x_n}_{n=0}^\infty$ and let $\set{y_n}_{n=0}^\infty$ be a $\kappa$-average-pseudo-orbit in $A$ obtained by application of Lemma~\ref{lematC}. Let $z\in A$ be a point which $\eps/4$-shadows it in average. Note that there is $M>0$ such that for every $n\geq M$ we have the inequalities
\begin{eqnarray*}
\frac{1}{n}\cdot\#\set{0\leq i< n :\dd(x_i,y_i)\geq \frac{\eps}{4}} &<& \frac{\eps}{4},\\
\frac{1}{n}\sum_{i=0}^{n-1}\dd(f^i(z),y_i)&<&\frac{\eps}{4}.
\end{eqnarray*}
Thus if we denote $\Gamma_n=\set{i<n\colon\dd(x_i,y_i)< \frac{\eps}{4}}$ we have
\begin{eqnarray*}
\frac{1}{n}\sum_{i=0}^{n-1} \dd(f^i(z),x_i)&\leq&\frac{n-\#\Gamma_n}{n}+\frac{1}{n}\sum_{i\in \Gamma_n}\left( \dd(f^i(z),y_i)+\dd(y_i,x_i)\right)\\
&\leq& \frac{\eps}{4}+\frac{1}{n}\sum_{i=0}^{n-1} \dd(f^i(z),y_i)+\frac{\eps}{4}\leq\frac{\eps}{4}+\frac{\eps}{4}+\frac{\eps}{4}.
\end{eqnarray*}
This proves that $\limsup_{n\ra \infty} \frac{1}{n}\sum_{i=0}^{n-1}\dd(f^i(z),x_i) \leq 3/4\cdot\eps<\eps$. Consequently, $\set{x_n}_{n=0}^\infty$ is $\eps$-shadowed in average by $z$ and the proof is completed.
\end{proof}

%%%%%%%%%%%%%%%%%%%%%%%%%%%%%%%%%%%%%%%%%%%%%%%%%%%%%%%%%%%%%%%%%%%%%%%%%%%%%%%%%%%%%%%%%%%%%%%%%%%%%%%%%%%%%%%%%%%%%
\section{Factors and equicontinuity}
We do not know (see Question \ref{q:factors}) if a factor of a system
with the (asymptotic) average shadowing property has to have this property.
If the answer is yes, then the following theorem follows from \cite[Corollary 3.8]{Niu}.

\begin{thm}\label{thm:equi_asp}
Suppose that a compact dynamical system $(X,f)$ has the average shadowing property. Then if $(Y,\dd')$ is a metric space and $(Y,g)$ is a maximal equicontinuous factor of $(X,f)$, then $Y$ is a singleton.
\end{thm}
\begin{proof}
Let $\pi \colon X \to Y$ be a factor map onto a maximal equicontinuous factor. If $Y$ has at least two elements, then there are
$x,y\in X$ such that $\pi(x)\neq \pi(y)$. Since $(Y,g)$ is equicontinuous, it is distal, and there is a $\lambda>0$ such that
$\liminf_{n\to \infty} \dd'(g^n(\pi(x)),g^n(\pi(y)))=2\lambda>0$. By equicontinuity there is $\eps>0$ such that if $\dd(p,q)<\eps$, then $\dd(f^n(p),f^n(q))<\lambda/3$ for every $n$. Using compactness we can find a $\delta>0$ such that if $\dd(p,q)<\delta$ then $\dd'(\pi(p),\pi(q))<\eps$.
If we fix any $\eta>0$, then we can find a sufficiently large $N$ such that the sequence
$$
\xi=x,f(x),\ldots,f^{N-1}(x),f^N(y), f^{N+1}(y),\ldots, f^{2N-1}(y),f^{3N}(x),f^{3N+1}(x),\ldots, f^{4N-1}(x),f^{4N}(y),\ldots
$$
is an $\eta$-average pseudo-orbit.
In particular, for sufficiently small $\eta$ (and large $N$) there is a point $z$ which $\delta$-shadows it in average.
Note that the set of all indices $j$ such that $\xi_j=f^l(x)$ for some $l\ge 0$ %which coincide with the fragments of orbit of $x$ (or $y$)
has upper density at least $1/2$, and an analogous statement holds with $x$ replaced by $y$. Therefore there
are $0\leq i < j$ such that $\dd(f^i(z),f^i(x))<\delta$ and $\dd(f^j(z),f^j(y))<\delta$.
This implies that $\dd'(g^i(\pi(z)),g^i(\pi(x)))<\eps$ and $\dd'(g^j(\pi(z)),g^j(\pi(y)))<\eps$.
In particular, for all sufficiently large $n$, we have
$$
\lambda<\dd'(g^n(\pi(x)),g^n(\pi(y))) \leq \dd'(g^n(\pi(x)),g^n(\pi(z)))+\dd'(g^n(\pi(z)),g^n(\pi(y)))\leq 2\lambda/3<\lambda,
$$
which is a contradiction.
\end{proof}

Using the above theorem or the well known fact that every distal system has fully supported invariant measure and Theorem \ref{thm:measure+asp-imply-wm}
one obtains another proof of \cite[Corollary 3.8]{Niu} which says that the only compact, distal dynamical system with the average shadowing property is a trivial one-point system.
%\begin{thm}\label{thm:distal_asp}
%If $(X,f)$ is compact, distal, and has the average shadowing property, then $X$ is a singleton.
%\end{thm}
%\begin{proof}
%If $(X,f)$ is distal then it is a union of disjoint distal systems. Since by Theorem~\ref{thm:equi_asp} maximal equicontinuous factor of $(X,f)$
%is trivial, each of these systems has to be singleton, since otherwise, as a weakly mixing system, it has to have a non-diagonal proximal pair, which is %impossible. But then $X$ is a union of trivial systems (fixed points) and hence must be a single point, since it has the average shadowing property.
%\end{proof}

%%%%%%%%%%%%%%%%%%%%%%%%%%%%%%%%%%%%%%%%%%%%%%%%%%%%%%%%%%%%%%%%%%%%%%%%%%%%%%%%%%%%%%%%%%%%%%%%%%%%%%%%%%%%%%%%%%%%%
\section{Consequences of asymptotic shadowing}

Limit shadowing property was introduced by Eirola, Nevanlinna, and Pilyugin in \cite{Ei2}.
It is known that limit shadowing property is always present in a neighborhood of a hyperbolic set for a diffeomorphism $f$ of $\R^n$ (see \cite[Theorem 1.4.1]{Pil}). Limit shadowing property is an important notion with many possible applications (see \cite{Pil}).
It is also known that $f$ has shadowing property in a neigborhood of its hyperbolic set.
Our Theorem~\ref{limsh_chaintrans} below shows that such a situation is more general than it seems at the first sight.

\begin{defn}
A compact dynamical system $(X,f)$ has \emph{limit shadowing} if every asymptotic pseudo-orbit $\{x_i\}_{i=0}^\infty$ is \emph{asymptotically shadowed} by some point $z\in X$, that is
$$\lim_{n\rightarrow\infty}\dd(f^i(z),x_i)=0.$$

We say that the compact system $(X,f)$ has \emph{s-limit shadowing} if it has shadowing and for every $\eps>0$ there is $\delta>0$ such that every $\delta$-pseudo-orbit that is also an asymptotic pseudo-orbit is both $\eps$-shadowed and asymptotically shadowed by some point $z\in X$.
\end{defn}

\begin{rem}
If $f$ is surjective, then s-limit shadowing clearly implies limit shadowing. There are, however, examples
of dynamical systems with shadowing and limit shadowing but without s-limit shadowing (see for instance \cite[Example 3.5]{BGO}).
\end{rem}

\begin{thm}\label{limsh_chaintrans}
Assume that a compact dynamical system $(X,f)$ is chain transitive. If $(X,f)$ has limit shadowing then it also has shadowing.
\end{thm}
\begin{proof}
Suppose on the contrary that $(X,f)$ does not have shadowing. Then there is $\eps>0$ such that for any $n>0$ there is a finite $\frac{1}{n}$-pseudo-orbit $\alpha_n$ which cannot be $\eps$-shadowed by any
point in $X$. Using chain transitivity, for every $n$ there exists a $\frac{1}{n}$-pseudo-orbit $\beta_n$ such that the sequence
$\alpha_n\beta_n\alpha_{n+1}$ forms a $\frac{1}{n}$-pseudo-orbit. Then the infinite concatenation
$$
\alpha_1 \beta_1 \alpha_2 \beta_2 \alpha_3 \beta_3 \ldots
$$
is an asymptotic pseudo orbit, and therefore, by limit shadowing, it is asymptotically shadowed by some point $z\in X$. Asymptotic shadowing is also, starting at some point, $\eps$-shadowing, so this would mean that almost all $\alpha_n$ are $\eps$-shadowed by some point of the form $f^{i_n}(z)$, which is a contradiction.
\end{proof}

\begin{rem}
It is easy to verify that if $(X,f)$ is chain transitive and has shadowing or limit shadowing then it is transitive.
\end{rem}

By the above facts, we can extend results of \cite{LiSakai} to a complete characterization of shadowing in expansive systems.
\begin{cor}
Let $(X,f)$ be a compact, c-expansive, and transitive dynamical system. Then the following conditions are equivalent:
\begin{enumerate}
\item $(X,f)$ has shadowing,
\item $(X,f)$ has limit shadowing,
\item $(X,f)$ has s-limit shadowing.
\end{enumerate}
\end{cor}

%%%%%%%%%%%%%%%%%%%%%%%%%%%%%%%%%%%%%%%%%%%%%%%%%%%%%%%%%%%%%%%%%%%%%%%%%%%%%%%%%%%%%%%%%%%%%%%%%%%%%%%%%%%%%%%%%%%%%
\section{Examples}\label{sec:examples}

In this section we present some examples of systems with the almost specification property (average and asymptotic average shadowing properties). They show that there is no clear relation between these properties and standard notions of recurrence such as transitivity, mixing, etc.

\begin{exmp}\label{ex:KuOp}
It was proved in \cite[Example 3.13]{KuOp2} that there exist non-transitive systems with the asymptotic average shadowing property.
The same reasoning as in \cite{KuOp2}, but with Theorem \ref{thm:almostspec_measure} in place of \cite[Theorem 3.11]{KuOp2} shows that dynamical systems described in \cite[Example 3.13]{KuOp2} have the almost specification property.
By Theorem~\ref{thm:aasp-implies-asp} this implies that there also exist non-transitive systems with the  average shadowing property.
\end{exmp}

Motivated by the Example \ref{ex:KuOp} and with the help of Lemma \ref{lem:unique_erg+prox-imply-almost} we are able to provide numerous examples of proximal systems with different degrees of transitivity. All these systems are proximal and uniquely ergodic.
Moreover, all these systems have the almost specification property (hence average and asymptotic average shadowing properties). This shows, in contrast to Theorem~\ref{thm:summary}, that neither of these properties is strong enough to induce any form of transitivity when only weaker kind of transitivity is assumed and the measure center is a proper subsystem (e.g., weak mixing and the average shadowing property are not sufficient to induce mixing, etc.).

\begin{lem}\label{lem:unique_erg+prox-imply-almost}
Every uniquely ergodic and proximal compact dynamical system has the almost specification property.
\end{lem}
\begin{proof} Unique ergodicity means that there exists only one invariant measure. Proximality implies that
the system has a fixed point. The atomic measure concentrated on a fixed point is an invariant measure. Therefore the measure center of uniquely ergodic and proximal compact dynamical system is a singleton. It contains only one point --- the unique fixed point of the system. Since the trivial one-point system has the almost specification property we may apply  Theorem \ref{thm:almostspec_measure} to complete the proof.
\end{proof}

\begin{exmp}
Let $(\Sigma,\sigma)$ be the subshift of the full shift on symbols $\set{0,1}$ such that $x\in \Sigma$ if and only if $\#(\set{j : x_j=1}\cap [i,i+2^n])\leq n$ for every $i,n>0$. By \cite[Theorem 7.1]{KwSpacing}, $(\Sigma,\sigma)$ is uniquely ergodic, proximal, and has entropy zero. In particular, it does not have the specification property.
%It is easy to see that $\sigma$ is mixing on $\Sigma$. Fix any word $u$ from the language of $\Sigma$ and containing at least one occurrence of $1$. Let $C_\Sigma[u]$ be the cylinder determined by $u$ in $\Sigma$. Then $N(x,C_\Sigma[u])\subset k+N(x,C_\Sigma[u])$ for some $k\geq 0$, therefore $d^*(N(x,C_\Sigma[u]))=0$. %But it has almost specification by Lemma \ref{lem:unique_erg+prox-imply-almost}.
\end{exmp}

\begin{rem}
There is also a large class of dynamical systems with almost specification property (but without specification).
In particular it contains all $\beta$-shifts (e.g. see \cite{T}).
\end{rem}

\begin{exmp} Let $P\subset\N$ be any thick set with thick complement. Let $(\Omega_P,\sigma_P)$ denote the spacing shift on two symbols generated by $P$ (see \cite{spacing}, or \cite{KwSpacing,KwOpETDS} for definition). It was proved in \cite{KwSpacing} that $\Omega_P$ is uniquely ergodic (with the unique ergodic measure concentrated on $0^\infty$), and hence proximal exactly when the entropy of $\sigma_P$ is zero. By the results of \cite{spacing}, if $P$ has thick complement then the entropy of $\sigma_P$ is zero. Therefore $(\Omega_P,\sigma_P)$ is a weakly mixing, uniquely ergodic, and proximal dynamical system which is not mixing.
\end{exmp}

\begin{exmp}
In \cite{spacing} and \cite{KwOpETDS} the authors constructed $P$ such that $(\Omega_P,\sigma_P)$ is totally transitive, but not weakly mixing. In their example $P$ has thick complement, therefore the system is proximal and uniquely ergodic.
\end{exmp}

\begin{exmp} Let $P\subset\N$ be any thick set with thick complement. Define $Y = \Omega_P \times\{-1,1\}$. Let $g\colon Y\to Y$ be defined by $g(x,i)=(\sigma_P(x),-i)$. Let $X$ be obtained from $Y$ by collapsing $\set{(0^\infty,-1),(0^\infty,1)}$ to a single point $p$ and let $(X,f)$ be obtained from $(Y,g)$ by the natural projection. Then $f$ is transitive, but not totally transitive. It is also easy to see that the only invariant ergodic measure for $(X,f)$ must be concentrated on $p$, and therefore $(X,f)$ is uniquely ergodic and proximal.
\end{exmp}

\section{The noncompact case}
The next two examples illustrate that without the assumption of compactness of $X$ there is no implication in either direction between the average and the asymptotic average shadowing properties.

\begin{exmp}\label{example1}
Let $p$, $a_0$ and $b_0$ be vertices of an equilateral triangle with side $1$. For $n\in\{-1,-2,\ldots\}$ define inductively $a_n$ to be the middle of the line segment $\overline{pa_{n+1}}$ and $b_n$ to be the middle of the line segment $\overline{pb_{n+1}}$. For every $n\in\Zp$ choose $a_n$ and $b_n$ so that $\{a_n\}_{n\in\Z}$ and $\{b_n\}_{n\in\Z}$ are two disjoint sequences and $p$ is not a term in either of them. Define $X_1=\{p\}\cup\bigcup_{n\in\Z}\{a_n,b_n\}$. We define a metric $\dd_1$ on $X_1$ in the following way:
\begin{enumerate}[(i)]
\item\label{e1:c1} if $x,y\in\{p\}\cup\bigcup_{n\in\{0,-1,\ldots\}}\{a_n,b_n\}$ then $\dd_1(x,y)$ is the Euclidean distance between $x$ and $y$ in the triangle,
\item\label{e1:c2} for every $n\in\N$ we put $\dd_1(a_n,a_{n+1})=\dd_1(b_n,b_{n+1})=2^n$,
\item\label{e1:c3} for every $n\in\Zp$ we put $\dd_1(a_n,b_n)=\sum_{i=1}^n1/i$,
\item\label{e1:c4} for every pair of points $x,y\in X_1$ for which $\dd_1(x,y)$ has not been defined in steps \eqref{e1:c1}--\eqref{e1:c3} we put
$$\begin{array}{rcl}\dd_1(x,y)&=&\inf\left\{\sum_{i=0}^n\dd_1(c_i,c_{i+1})\:|\:n\in\N, c_0,\ldots,c_{n+1}\in X_1, c_0=x, c_{n+1}=y,\right.\\&&\left.\dd_1(c_0,c_1),\ldots,\dd_1(c_n,c_{n+1})\mbox{ have been defined in steps \eqref{e1:c1}--\eqref{e1:c3}}\right\}.\end{array}$$
\end{enumerate}
It is elementary to verify that $\dd_1$ is a metric. Finally, define $f_1\colon X_1\to X_1$ by setting $f_1(p)=p$ and by putting $f_1(a_n)=a_{n+1}$ and $f_1(b_n)=b_{n+1}$ for every $n\in\Z$. It can be easily verified that $f_1$ is continuous.
\end{exmp}

%197mm in 24pt=82 in 10pt
\begin{figure}
\includegraphics[width=82mm]{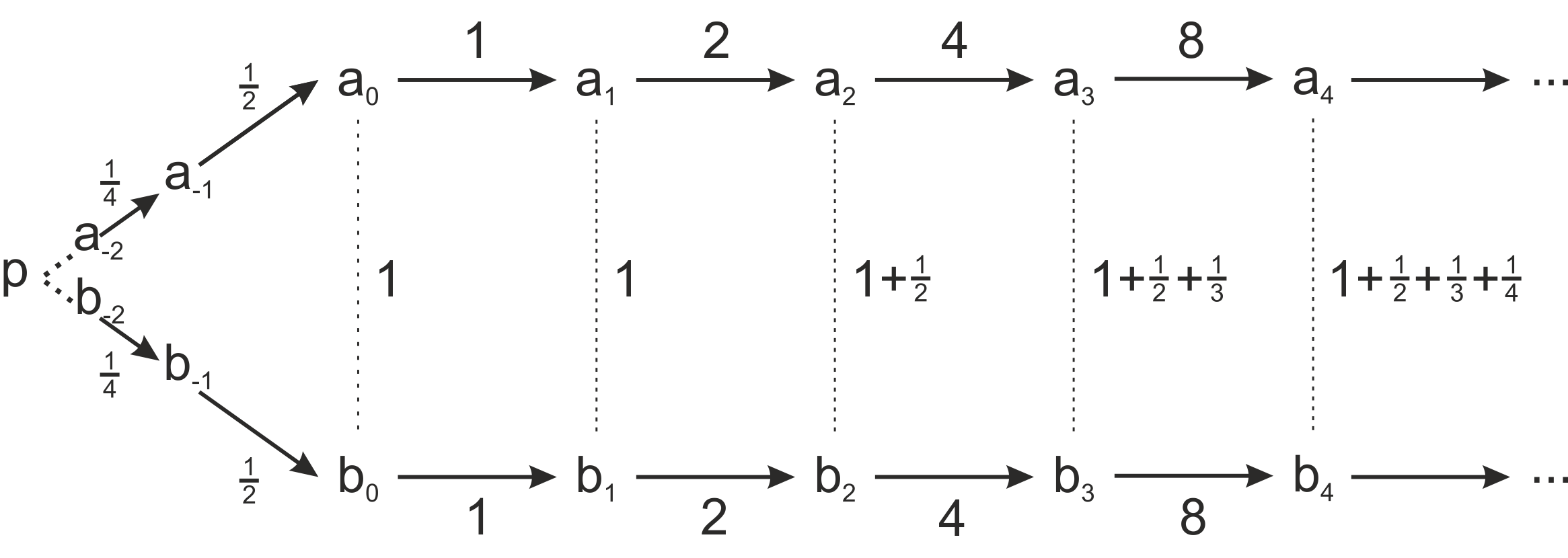}\caption{The space $X_1$ and the map $f_1$}
\end{figure}

\begin{thm}
The map $f_1$ has the average shadowing property, but not the asymptotic average shadowing property.
\end{thm}
\begin{proof}
\emph{Average shadowing property.} Given $\varepsilon>0$ set $\delta<\varepsilon/3$. For every $\delta$-average pseudo-orbit $\{z_i\}_{i=0}^\infty$ with a constant $N$ as in Definition~\ref{def:aver_po}, there are two possibilities:

\noindent \emph{Case 1.} The sequence $\{z_i\}_{i=0}^\infty$ is bounded. Consider all the indices $r_i$ where $\dd_1(p,z_{r_i+1})\leq \dd_1(p,z_{r_i})$. As the sequence $\{\dd_1(p,z_i)\}_{i=0}^\infty$ is bounded and takes values from the set $\{2^i\colon i\in\Z\}$ these indices can be written as an infinite increasing sequence $\{r_i\}_{i=0}^\infty$. Put $r_{-1}=-1$. Observe that the structure of $X_1$ guarantees that for every $i\in\N$ we have
$$
\sum_{k=r_{i-1}+1}^{r_i}\dd_1(p,z_k)\leq\dd_1(p,f_1(z_{r_i})) \;\leq\; 2\dd_1(f_1(z_{r_i}),z_{r_i})\leq 2\dd_1(f_1(z_{r_i}),z_{r_{i+1}})
$$
which implies that for every $n\in\N$ we have
$$\frac{1}{r_n+1}\sum_{i=0}^{r_n}\dd_1(p,z_i)\leq\frac{2}{r_n+1}\sum_{i=0}^{r_n}\dd_1(f_1(z_i),z_{i+1}).$$
Note that since sequence $\{z_i\}_{i=0}^\infty$ is bounded, the sums $\sum_{k=r_{i-1}+1}^{r_i}\dd_1(p,z_k)$ are also bounded by some constant $R>0$. Take $n\geq N$ large enough for the inequality $R/(n+1)\leq\delta$ to hold and such that $n>r_0$. Let $r_j$ be the largest term of the sequence $\{r_i\}_{i=0}^\infty$ smaller than $n$. Finally observe that
\begin{eqnarray*}
\frac{1}{n+1}\sum_{i=0}^n\dd_1(p,z_i)&=&\frac{1}{n+1}\sum_{i=0}^{r_j}\dd_1(p,z_i)+\frac{1}{n+1}\sum_{i=r_j+1}^n\dd_1(p,z_i)\\
&\leq&\frac{2}{n+1}\sum_{i=0}^n\dd_1(f_1(z_i),z_{i+1})+\frac{R}{n+1}\leq 3\delta<\varepsilon
\end{eqnarray*}
which proves $\varepsilon$-shadowing in average by $p$.

\noindent \emph{Case 2.} The pseudo-orbit $\{z_i\}_{i=0}^\infty$ is not bounded. Notice that it follows straight from the definition of the $\delta$-average-pseudo-orbit that the maximum possible error $\dd_1(f(z_i),z_{i+1})$ is bounded by $\delta N$. Therefore, once the pseudo-orbit strays far enough from $p$, it will be unable to deviate from the trajectory of the point it landed on and will be $\varepsilon$-shadowed in average by the orbit of that point.

\medbreak\noindent \emph{Asymptotic average shadowing property.} Define the sequence $\{z_i\}_{i=0}^\infty$ by
$$z_i=\left\{\begin{array}{cl}
b_i&\mbox{when } i\in\{0,1\},\\
a_i&\mbox{when } i\in(2^{2k},2^{2k+1}]\mbox{ for some }k\in\N,\\
b_i&\mbox{when } i\in(2^{2k+1},2^{2k+2}]\mbox{ for some }k\in\N.\\
\end{array}\right.$$
Observe that
\begin{eqnarray*}
\frac{1}{n}\sum_{i=0}^n\dd_1(f_1(z_i),z_{i+1})&=&\frac{1}{n}\sum_{i\in\{1,2,4,\ldots\}\cap[0,n]}\dd_1(f_1(z_i),z_{i+1})\\
&\leq&\frac{1}{n}(1+\log_2n)\sum_{i=1}^{n+1}\frac{1}{i}\\
&\leq&\frac{1+\log_2n}{\sqrt{n}}\sum_{i=1}^{n+1}\frac{1}{i\sqrt{i}}.
\end{eqnarray*}
Observe that this last expression tends to $0$ when $n\rightarrow+\infty$, proving that $\{z_i\}_{i=0}^\infty$ is an asymptotic average pseudo-orbit of $f_1$.

This pseudo-orbit may not be asymptotically shadowed in average by any point other than $a_0$ or $b_0$, as it maintains a distance bounded from below by $1/2$ from the iterations of any point but these two. To exclude the possibility of shadowing by $a_0$ note that for $n$ large enough we have $\frac{1}{2^{2n}}\sum_{i=0}^{2^{2n}}\dd_1(f_1^i(a_i),z_i)\geq\frac{1}{2^{2n}}2^{2n-1}=1/2$. The shadowing by $b_0$ is handled in an analogous way, proving that $f_1$ does not have the asymptotic average shadowing property.
\end{proof}

\begin{exmp}\label{example2}
First we define a family of pairwise disjoint metric spaces $\{(A_{i,j},\dd_{i,j})\}_{i\in\Z,j\in\N}$ in the following way:
\begin{enumerate}
\item\label{e2:ci} for every $i\in\Z_-$ and $j\in\N$ we put $A_{i,j}=\{a_{i,j}\}$,
\item\label{e2:cii} for every $j\in\N$ we put $A_{0,j}=\{a_{0,j},b_{0,j}\}$ and $\dd_{0,j}(a_{0,j},b_{0,j})=1/2^j$,
\item\label{e2:ciii} for every $i\in\Zp$ and $j\in\N$ we put $A_{i,j}=\{a_{i,j},b_{i,j},c_{i,j}\}$, $\dd_{i,j}(a_{i,j},b_{i,j})=1/2^j$, $\dd_{i,j}(b_{i,j},c_{i,j})=1-1/2^j$, and $\dd_{i,j}(a_{i,j},c_{i,j})=1$.
\end{enumerate}
Next we put $X_2=\bigcup_{i\in\Z,j\in\N}A_{i,j}$ and define a metric $\dd_2$ on $X_2$ in the following way:
\begin{enumerate}[(i)]
\item\label{e2:c0} for every $i\in\Z$, $j\in\N$, $x,y\in A_{i,j}$ we put $\dd_2(x,y)=\dd_{i,j}(x,y)$,
\item\label{e2:c1} for every $i,j\in\N$, $x\in A_{i,j}$ and $y\in A_{i+1,j}$ we put $\dd_2(x,y)=2^i$,
\item\label{e2:c2} for every $i\in\Z_-$, $j\in\N$, $x\in A_{i,j}$ and $y\in A_{i+1,j}$ we put $\dd_2(x,y)=2^{-i-1}$,
\item\label{e2:c3} for every $i\in\Z$, $j,k\in\N$, $x\in A_{i,j}$ and $y\in A_{i,k}$ we put $\dd_2(x,y)=2^{|i|}$,
\item\label{e2:c4} for every pair of points $x,y\in X_2$ for which $\dd_2(x,y)$ has not been defined in steps \eqref{e2:c0}--\eqref{e2:c3} we put
$$\begin{array}{rcl}\dd_2(x,y)&=&\inf\left\{\sum_{i=0}^n\dd_2(c_i,c_{i+1})\:|\:n\in\N, c_0,\ldots,c_{n+1}\in X_2, c_0=x, c_{n+1}=y,\right.\\&&\left.\dd_2(c_0,c_1),\ldots,\dd_2(c_n,c_{n+1})\mbox{ have been defined in steps \eqref{e2:c0}--\eqref{e2:c3}}\right\}.\end{array}$$
\end{enumerate}
Once again, it is elementary to check that $\dd_2$ is a metric. Finally, define $f_2\colon X_2\to X_2$ as follows:
\begin{enumerate}
\item for every $i\in\Z$ and $j\in\N$ put $f_2(a_{i,j})=a_{i+1,j}$,
\item for every $i,j\in\N$ put $f_2(b_{i,j})=c_{i+1,j}$,
\item for every $i\in\Zp$ and $j\in\N$ put $f_2(c_{i,j})=a_{i+1,j}$.
\end{enumerate}
Since the topology on $X_2$ is discrete the map $f_2$ is continuous.
\end{exmp}

\begin{figure}
\includegraphics[width=87.5mm]{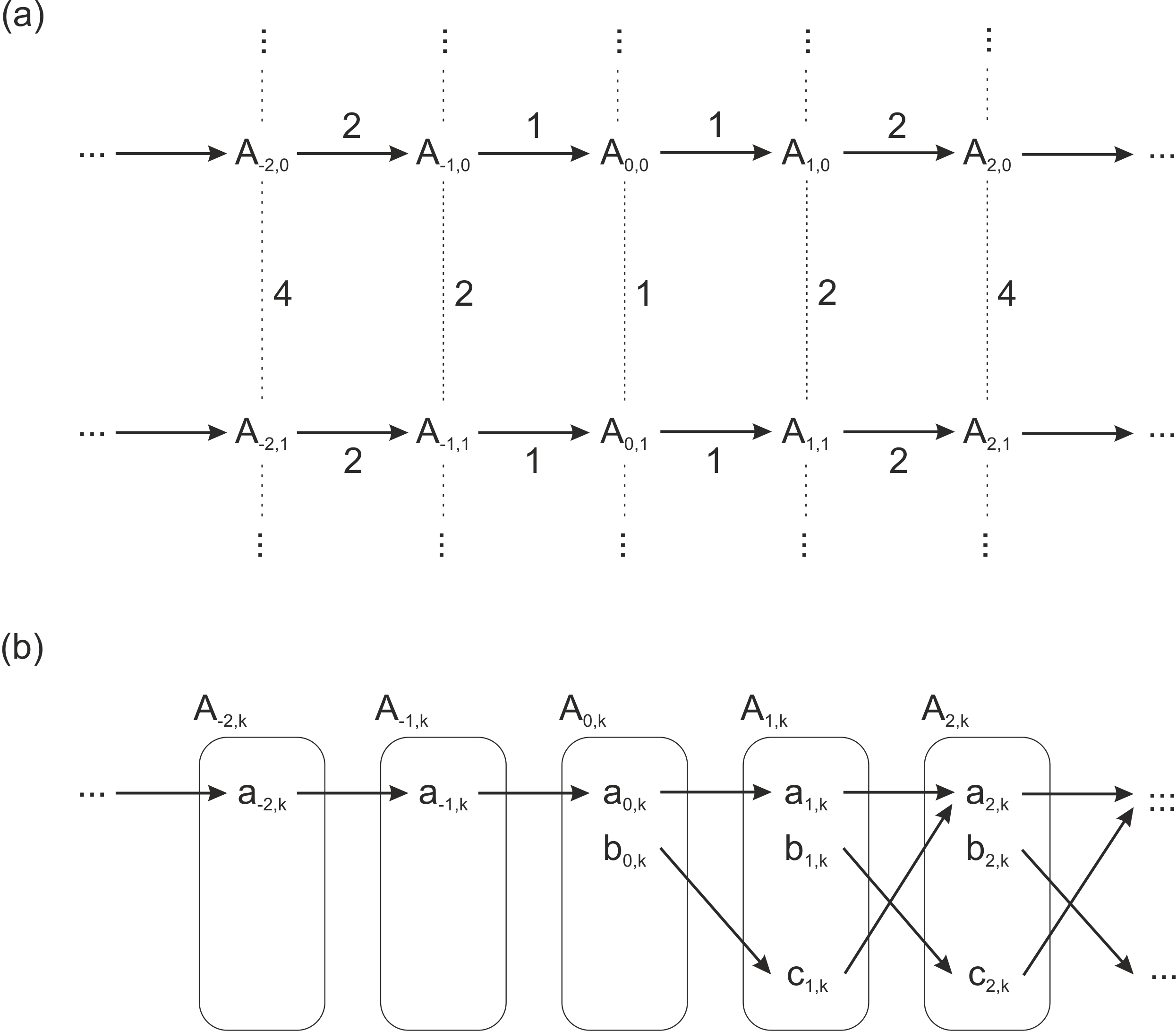}\caption{(a) The space $X_2$ and the map $f_2$; (b) a detailed view of one of the rows.}
\end{figure}

\begin{thm}
The map $f_2$ does not have the average shadowing property, but has the asymptotic average shadowing property.
\end{thm}
\begin{proof}
\emph{Average shadowing property.} Let $k\in\N$ be fixed. Consider the sequence $\{z_i\}_{i=0}^\infty=\{b_{0,k},c_{1,k},b_{2,k},c_{3,k},\ldots\}$. Since for every $k\in\N$ we have $\frac{1}{n}\sum_{i=0}^{n-1}\dd_2(f_2(x_{i+k}),x_{i+k+1})\leq\frac{n}{n2^k}=\frac{1}{2^k}$ this sequence is a $\frac{1}{2^{k-1}}$-average-pseudo-orbit. Notice that $\{z_i\}_{i=0}^\infty$ may not be shadowed in average by any point except for $a_{0,k}$ or $b_{0,k}$, as $\{z_i\}_{i=0}^\infty$ maintains a positive distance from the orbit of any other point. Note also that $\frac{1}{n}\sum_{i=0}^{n-1}\dd_2(f_2(a_{0,k}),z_i)\geq\lfloor\frac{n}{2}\rfloor/n\geq\frac{1}{4}$. The possibility of shadowing by $b_{0,k}$ is excluded in an analogous way. Therefore, for $\varepsilon=1/5$ there is no $\delta>0$ for which every $\delta$-average-pseudo-orbit would be $\varepsilon$-shadowed in average by some point in $X_2$.

\noindent\emph{Asymptotic average shadowing property.} Let $\{z_i\}_{i=0}^\infty$ be an asymptotic-average pseudo-orbit of $f_2$. Consider all the indices $r_i$ for which the points $f_2(z_{r_i})$ and $z_{r_i+1}$ belong to different spaces $A_{i,j}$. Let us first assume that there are infinitely many such indices. In this case these indices can be written as an infinite increasing sequence $\{r_i\}_{i=0}^\infty$. Put $r_{-1}=-1$. Observe that the structure of $X_2$ guarantees that for every $i\in\Zp$ we have $\dd_2(f_2(z_{r_{i-1}}),z_{r_{i-1}+1})+\dd_2(f_2(z_{r_i}),z_{r_i+1})\geq 2^{\lfloor(r_i-r_{i-1})/2\rfloor}$. Consequently for every $k\in\N$ we have
\begin{eqnarray}
&&\limsup_{k\rightarrow\infty}\frac{1}{r_k+1}\sum_{i=0}^{r_k}\dd_2(f_2(z_k),z_{k+1})\nonumber\\
&&\quad\quad\quad\geq\frac{1}{2}\limsup_{k\rightarrow\infty}\frac{1}{r_k+1}\left(\sum_{i=0}^{r_k}\dd_2(f_2(z_k),z_{k+1})+\sum_{i=1}^{r_k-1}\dd_2(f_2(z_k),z_{k+1})\right)\nonumber\\
&&\quad\quad\quad=\frac{1}{2}\limsup_{k\rightarrow\infty}\frac{1}{r_k+1}\sum_{i=1}^{r_k}\left(\dd_2(f_2(z_{k-1}),z_k)+\dd_2(f_2(z_k),z_{k+1})\right)\label{eq:e2:long}\\
&&\quad\quad\quad=\frac{1}{2}\limsup_{k\rightarrow\infty}\frac{1}{r_k+1}\sum_{i=0}^{r_k}\left(\dd_2(f_2(z_{k-1}),z_k)+\dd_2(f_2(z_k),z_{k+1})\right)\nonumber\\
&&\quad\quad\quad\geq\frac{1}{2}\limsup_{k\rightarrow\infty}\frac{1}{r_k+1}\sum_{i=0}^k2^{\lfloor(r_i-r_{i-1})/2\rfloor}\nonumber
\end{eqnarray}
By the fact that for every $n\in\N$ we have the inequality $2^{\lfloor n/2\rfloor}\geq n/2$ and the fact that the $k+1$ terms $r_0-r_{-1},\ldots,r_k-r_{k-1}$ add up to $r_k+1$, we obtain by \eqref{eq:e2:long} that
$$
\limsup_{k\rightarrow\infty}\frac{1}{r_k+1}\sum_{i=0}^{r_k}\dd_2(f_2(z_k),z_{k+1})\geq\frac{1}{2}\limsup_{k\rightarrow\infty}\frac{1}{r_k+1}\cdot\frac{r_k+1}{2}=\frac{1}{4}.
$$
Since this would contradict the fact that $\{z_i\}_{i=0}^\infty$ is an asymptotic average pseudo-orbit, we conclude that there are only finitely many indices $r_i$, and therefore there exist $p\in\Z$ and $q,k\in\N$ such that for every $n\geq k$ the points $f_2^n(a_{p,q})$ and $z_n$ belong to the same space $A_{p+n,q}$.

In order to prove that the point $a_{p,q}$ asymptotically shadows in average $\{z_i\}_{i=0}^\infty$ observe that every positive term in the sum $\sum_{i=k}^n\dd_2(f_2(z_i),z_{i+1})$ adds at most $2^q+1$ times its value to the sum $\sum_{i=k}^n\dd_2(f_2^i(a_{p,q}),z_i)$.
Strictly speaking, if $\dd_2(f_2(z_j),z_{j+1})>0$ for some $j\geq k$, then, assuming $f_2(z_{j+1})=z_{j+2}$, the following situations are possible:
$$\begin{array}{cc|c|cc}
f_2(z_j)&z_{j+1}&\dd_2(f_2(z_j),z_{j+1})&\dd_2(f_2^{j+1}(a_{p,q}),z_{j+1})&\dd_2(f_2^{j+2}(a_{p,q}),z_{j+2})\\
\hline
a_{p+j+1,q}&b_{p+j+1,q}&1/2^q&1/2^q&1\\
a_{p+j+1,q}&c_{p+j+1,q}&1&1&0\\
b_{p+j+1,q}&a_{p+j+1,q}&1/2^q&0&0\\
b_{p+j+1,q}&c_{p+j+1,q}&1-1/2^q&1&0\\
c_{p+j+1,q}&a_{p+j+1,q}&1&0&0\\
c_{p+j+1,q}&b_{p+j+1,q}&1-1/2^q&1/2^q&1\\
\end{array}$$
Starting from $z_{j+3}$ the sequence $\{z_i\}_{i=0}^\infty$ will follow the trajectory of $a_{p,q}$ until the next deviation.

If $f_2(z_{j+1})\neq z_{j+2}$ then the deviation at $j+1$ contributes only the fourth column of the above table (instead of fourth and fifth) to the total error $\sum_{i=k}^n\dd_2(f_2^i(a_{p,q}),z_i)$.

Therefore
\begin{eqnarray*}
\lim_{s\rightarrow\infty}\frac{1}{s}\sum_{i=0}^{s-1}\dd_2(f_2^i(a_{p,q}),z_i)&=&\lim_{s\rightarrow\infty}\frac{1}{s}\sum_{i=k}^{s-1}\dd_2(f_2^i(a_{p,q}),z_i)\\
&\leq&(2^q+1)\lim_{s\rightarrow\infty}\frac{1}{s}\sum_{i=k}^{s-1}\dd_2(f_2(z_i),z_{i+1})=0.
\end{eqnarray*}
This proves that the map $f_2$ has the asymptotic average shadowing property.
\end{proof}

\section{Open problems}
In this section we collect a few open questions for further research.

If $f\colon X\ra X$ is a continuous map of a compact metric space with the shadowing property, then so does $f$ restricted to its nonwandering set $\Omega(f)$, e.g. see \cite[Theorem 3.4.2]{AH} extended by \cite[Lemma 1]{Moothathu}. Motivated by this result and Theorem~\ref{thm:almostspec_measure} above we raise the following question.

\begin{que}\label{Q:measure_center}
Assume that $(X,f)$ has the almost specification (asymptotic average shadowing property, average shadowing property)
and let $A$ be any compact invariant set containing the measure center of $f$. Does the dynamical system $(A,f|_A)$ has the almost specification
(asymptotic average shadowing property, average shadowing property)?
\end{que}

Note that if the answer to the above question is positive, then we can get rid of surjectivity assumption from
Theorems \ref{thm:almost-spec-implies-aasp} and \ref{thm:aasp-implies-asp}

By Theorem~\ref{thm:equi_asp} we see that any minimal system with the average shadowing property has to be weakly mixing. We are, however, unable to provide any example of non-singleton minimal system with almost specification or generalized shadowing.

\begin{que}
Is there a nontrivial minimal system $(X,f)$ with the almost specification property or the (asymptotic) average shadowing property?
\end{que}

We proved that almost specification implies asymptotic average shadowing, which in turn implies average shadowing. It is natural to ask if the converse is true. We conjecture it is not.
\begin{que}\label{q:converse_implications}
Does the (asymptotic) average shadowing property imply the almost specification property?
Does the average shadowing property imply asymptotic average shadowing property?
\end{que}

On the other hand, we provided examples that when the space is not compact, then there are no implications between the average shadowing property and the asymptotic average shadowing property; their construction relied heavily on the fact that the space was unbounded. In this context it is natural to ask the following.
\begin{que}
Can examples analogous to Examples~\ref{example1} and \ref{example2} be presented in a bounded metric space?
\end{que}

In view of Theorems \ref{thm:measure+asp-imply-wm} and the well-known properties of maps with the specification property we conjecture that the answer to the following question is positive.
\begin{que}
Does the (asymptotic) average shadowing property or the almost specification property imply topological mixing or positive topological entropy provided that the measure center is the whole space?
\end{que}

It is easy to see that a factor of a system with (almost) specification property also has (almost) specification. Any sofic shift
which is not of finite type shows that a factor of a system with shadowing does not necessarily have shadowing. We do not know if every factor of a system
with the (asymptotic) average shadowing property also has this property. Any counterexample would show that (asymptotic) average shadowing property does not imply almost specification.

\begin{que}\label{q:factors}
Is the (asymptotic) average shadowing property or the almost specification property inherited by factors?
\end{que}

We are also unable to answer the following question related to Theorem \ref{thm:dense_minimal}.
\begin{que}\label{q:dense_minimal}
Is it true that if a compact system has the (asymptotic) average shadowing property and an invariant measure of full support, then
the minimal points are dense?
\end{que}

%%%%%%%%%%%%%%%%%%%%%%%%%%%%%%%%%%%%%%%%%%%%%%%%%%%%%%%%%%%%%%%%%%%%%%%%%%%%%%%%%%%%%%%%%%%%%%%%%%%%%%%%%%%%%%%%%%%%%
\section*{Acknowledgements}

Research of Piotr Oprocha was supported by Narodowe Centrum Nauki (National Science Center) in Poland, grant no. DEC-2011/03/B/ST1/00790.
Research of Dominik Kwietniak was supported by the Polish Ministry of Science and Higher Education grant Iuventus Plus IP~$2011028771$.

%%%%%%%%%%%%%%%%%%%%%%%%%%%%%%%%%%%%%%%%%%%%%%%%%%%%%%%%%%%%%%%%%%%%%%%%%%%%%%%%%%%%%%%%%%%%%%%%%%%%%%%%%%%%%%%%%%%%%

\end{document}